\documentclass[psamsfonts]{amsart}

\usepackage{amssymb,amsfonts}
\usepackage[all,arc]{xy}
\usepackage{enumerate}
\usepackage{mathrsfs}
\usepackage{amsmath}
\usepackage{bbm}
\usepackage{tikz}
\usepackage{relsize}
\usepackage{pgfplots}
\usepackage{hyperref}
\usepackage[alphabetic]{amsrefs}
\usepackage{mathtools}
\usepackage{thmtools}
 \pgfplotsset{compat=1.18}
 \usetikzlibrary{pgfplots.groupplots}
\usetikzlibrary{arrows.meta}

\newtheorem{thm}{Theorem}
\newtheorem{cor}[thm]{Corollary}

\newtheorem{lem}[thm]{Lemma}

\theoremstyle{definition}
\newtheorem{defn}[thm]{Definition}

\newtheorem{claim}{Claim}



\theoremstyle{remark}

\newcommand{\Q}{\mathbb{Q}}

\newcommand{\R}{\mathbb{R}}

\newcommand{\e}{\varepsilon}

\newcommand{\C}{\mathbb{C}}

\newcommand{\Folner}{F\o lner}

\newcommand{\Z}{\mathbb{Z}}
\newcommand{\N}{\mathbb{N}}

\newcommand{\E}{\mathbb{E}}

\newcommand{\oh}{{\rm o}}

\usepackage{cancel}

\newcommand{\ch}{\mathbbm{1}}
\newcommand{\vertiii}[1]{{\left\vert\kern-0.25ex\left\vert\kern-0.25ex\left\vert #1 
		\right\vert\kern-0.25ex\right\vert\kern-0.25ex\right\vert}}

\newcommand{\M}{\mathcal{M}}
\newcommand{\D}{\mathbb{D}}

\newcommand{\cesE}{\mathop{\mathbb{E}}\displaylimits}
\newcommand{\cpt}{{cpt}}
\newcommand{\wm}{{wm}}
\newcommand{\aper}{{aper}}
\newcommand{\pret}{{pret}}
\newcommand{\exact}{{exact}}
\newcommand{\ext}{\mathrm{ext}}
\newcommand{\apx}{\mathrm{apx}}


\usepackage[letterpaper,margin=1.25in]{geometry}
\usepackage{microtype,needspace}
\allowdisplaybreaks[1]
\hypersetup{hidelinks,pdftitle={Monochromatic almost-Pythagorean quadruples: current draft}}
\title{Monochromatic Almost-Pythagorean Quadruples}
\author{James Leng, Redmond McNamara, and Andreas Mountakis}
\date{}

\renewcommand{\pmod}[1]{\text{ }(\mathrm{mod}\ #1)}

\begin{document}
\maketitle
\begin{abstract}
We prove that for any set $A$ with positive multiplicative upper density, there exist $x$, $y$, $z$ and $w$ in $A$ such that  
$x^2 + y^2w^2 = z^2$, generalizing a result of Frantzikinakis, 
Klurman and Moreira. 
Our techniques apply more broadly to a wide class of systems of equations, allowing us to solve these equations simultaneously and furthermore to restrict the w variable in an entirely different multiplicatively dense set. To do this, we prove a novel multiple recurrence result for multiplicative dynamical systems.
\end{abstract}

\section{Introduction}
An equation is called \emph{partition regular} if, for every 
partition of the natural numbers into finitely many colors 
$\N = C_1 \cup \cdots \cup C_r$ there exists one color $C_i$ 
and variables, in this case $x$, $y$ and $z$, in $C_i$ that 
satisfy the equation. Partition regularity is a fundamental 
question in Ramsey theory and combinatorics more broadly, see 
\cite{Rado-original,dinasso-luperi}. In \cite{Erdos-Graham}, Erd\H{o}s and Graham asked whether the Pythagorean
equation
$$x^2 + y^2 = z^2$$
is partition regular. In \cite{Fra-Klu-Mor}, Frantzikinakis, Klurman and Moreira showed that the equation \begin{equation}x^2 + n^2 = z^2\end{equation} is partition regular if you only require that $x$ and $z$ are the same color but not $n$. 

Separately, Brauer's theorem \cite{Brauer1928} implies that \begin{equation} \label{BH} x + yw = z \end{equation} is partition regular (see also \cite{Bergelson2010,Hindman2011}). In this paper, we generalize both of these results.
\begin{thm}\label{intro:monochromatic}
The equation \begin{equation} \label{bigeq1} x^2 + y^2 w^2 = z^2 \end{equation} is partition regular.
\end{thm}
This theorem generalizes \cite{Fra-Klu-Mor} because it says not only can we find $x$ and $z$ which are a leg and the hypotenuse of a Pythagorean triple of the same color, we can also find $y$ and $w$ of the same color whose product is the other leg. And it generalizes the partition regularity of \eqref{BH} because it says even if we don’t color all the natural numbers and only color the perfect squares, we can still find a solution to \eqref{BH} among the perfect squares.

In fact, we can generalize \eqref{bigeq1} in more ways. We will say an equation is density regular if for any set $A$ with positive upper density we can find a solution to the equation where all of the variables lie in $A$. This is more restrictive than being partition regular because every partition will have at least one cell with positive upper density. In this paper, when we talk about density we will mean multiplicative density (see \eqref{upperdensity} for the definition). Our proof uses the techniques of ergodic theory and dynamical systems and therefore can establish the density regularity of \eqref{bigeq1}.

Furthermore, we can also generalize density regularity as follows: instead of fixing just one set with positive upper density, fix two, call them $A$ and $B$. Then we will show that there exist $x$, $y$, and $z$ in $A$ and $w$ in $B$ such that \eqref{bigeq1} holds. Our method also applies to a host of other equations. Thus, for example, we have the theorem
\begin{thm}\label{Pythagorean}\label{intro:density}
Let $a, b \in \N$ and let $A$ and $B$ be two sets of positive upper multiplicative density. Then there exist $x$, $y$ and $z$ in $A$ and $w$ in $B$ such that
\[
a^2 x^2 + b^2 y^2 w^2 = a^2 z^2.
\]
Applying this to the case of $a = b = 1$, we recover Theorem \ref{intro:monochromatic}.
\end{thm}

We can also establish density regularity for 
\[
(w-y-xz)(w+\lambda y-xz)=(\lambda+1)^2(xz)^2.
\]
and 
\[
(y-w+xz)(b(b+1)y-w+xz)=(b(b+1)-1)^2(xz)^2.
\]
for any fixed coefficients $\lambda$ and $b$ in $\N$ provided in the second case that $b \geq 2$. These examples and others are worked out in the appendix. In each case, we use that there is a parameterization of solutions to each of these equations in terms of polynomials. For instance, for any natural numbers $k$, $n$ and $m$ with $n>m$, one can check that
\[
k(n^2 - m^2), \ \ 2kmn, \ \ k(n^2 + m^2)
\]
form a Pythagorean triple.
Then each of the partition regularity results follows as a corollary from our main theorem. In order to state the theorem, we need the following somewhat technical definition. This definition essentially says that the polynomials we can consider as part of our parameterization cannot have any modular obstructions or obstructions in the form of level sets of the Archimedean character $n \mapsto n^{it}$ to our method working.
\ifdefined\GoodRecurrenceDefined
\else
\gdef\GoodRecurrenceDefined{}
\begin{defn}\label{def:good-multiplicative-recurrence}
Let $P_1,P_2,P_3\in\Z[n,m]$ be nonconstant polynomials.
We say that the ordered triple $(P_1,P_2,P_3)$ is
\emph{good for multiplicative recurrence} if the following
conditions hold.

\begin{enumerate}
\item \emph{Factorization.} The three polynomials are homogeneous
of the same degree $D$. The first two are products of integer
linear forms, while the third may also contain one irreducible
binary quadratic form:
\[
P_j=\prod_{\nu=1}^{D}L_{j,\nu}\;(j=1,2)
\text{ \ \ and \ \ }
P_3=R^{\gamma}\prod_{\nu=1}^{D-2\gamma}L_{3,\nu},
\]
where $\gamma\in\{0,1\}$, each $L_{j,\nu}(n,m)=an+bm$
has integer coefficients not both zero, and $R\in\Z[n,m]$
is a homogeneous quadratic form irreducible over $\Q$.
When $\gamma=0$, the factor $R$ is omitted.

\item \emph{Linear factors in the ratios.} Let
$\widetilde P_3=\prod_{\nu=1}^{D-2\gamma}L_{3,\nu}$ be the
product of the linear factors of $P_3$, with the empty product
equal to $1$. Cancel proportional linear factors in
\[
\frac{P_1}{\widetilde P_3}
\text{ \ \ and \ \ }\frac{P_2}{\widetilde P_3}.
\]
Both reduced ratios must be nonconstant. The first must contain
a linear factor with multiplicity one in its numerator or
denominator. The second must also contain such a factor, and
that factor must be linearly independent of every linear factor
remaining in the first ratio. Here two linear forms are linearly
independent if neither is a rational multiple of the other.

\item\label{item:good-congruences} \emph{Congruences.} Let $\omega_j$ be the number of
nonconstant irreducible factors of $P_j$, counted with multiplicity. Thus
$\omega_1=\omega_2=D$ and $\omega_3=D-\gamma$.
There is an integer $W\geq1$ with the following property.
Consider integers $Q,q_1,q_2,q_3\in\N$ satisfying
\[
W\mid q_1,\;q_j^{\omega_j}\mid Q\;(j=1,2,3)
\text{ \ \ and \ \ }\gcd(q_j,q_k)=1\;(j\ne k),
\]
and, if $\gamma=1$, every prime divisor $p$ of $q_3$ admits
a solution to
\begin{equation}\label{eq:good-root-condition}
R(1,a)\equiv0\pmod p.
\end{equation}
We call such choices \emph{admissible}. For every admissible
choice there must be $v,w\in\N$ such that
\[
\gcd(P_j(v,w),Q)=q_j\;(j=1,2,3)
\]
and
\[
P_1(v,w)\equiv q_1\pmod{q_1^2}
\text{ \ \ and \ \ }
P_2(v,w)\equiv P_3(v,w)\equiv1\pmod{q_1}.
\]
\item \emph{Sectors.} There is a $\kappa>0$ such that
$P_j(x,1)>0$ for $x>\kappa$ and $j=1,2,3$.
For every finite nonempty set $\Theta\subseteq\R$ and every
$\delta>0$, there are constants $\kappa<a<b<\infty$ and
$\eta>0$ such that, for each $\ell\in\N$, we can choose an
open interval $J_{\ell,\delta}\subseteq[a,b]$ of length at
least $\eta$ on which
\[
\left|\log\frac{P_2(x,1)}{P_3(x,1)}\right|\leq\delta
\text{ \ \ and \ \ }
\max_{t\in\Theta}
\left|\left(\frac{P_1(x,1)}{\ell P_3(x,1)}\right)^{it}-1\right|
\leq\delta.
\]
The constants $a,b,\eta$ are independent of $\ell$; they may
depend on the triple, $\Theta$, and $\delta$.
\end{enumerate}

A collection $(P_{i,1},P_{i,2},P_{i,3})_{1\leq i\leq r}$
is good for multiplicative recurrence if each of its triples
is good in this sense. The degrees, quadratic forms, and
constants in the definition may depend on $i$.
\end{defn}
\fi
We are now in a place where we can state our main theorem
\begin{thm}\label{intro:general}
\xdef\GeneralRecurrenceTheoremNumber{\arabic{thm}}
Let $(P_{i,1},P_{i,2},P_{i,3})_{1\leq i\leq r}$ be a collection
that is good for multiplicative recurrence in the sense of
Definition~\ref{def:good-multiplicative-recurrence}. Suppose that
$A\subseteq\N$ and $B\subseteq\N^r$ have positive multiplicative
upper density. Then there are $a\in A$,
$\ell=(\ell_1,\ldots,\ell_r)\in B$, and $n_i,m_i\in\N$ such that
\[
a\frac{P_{i,2}(n_i,m_i)}{P_{i,3}(n_i,m_i)}\in A
\text{ \ \ and \ \ }
a\frac{P_{i,1}(n_i,m_i)}{\ell_iP_{i,3}(n_i,m_i)}\in A
\]
for every $1\leq i\leq r$.
\end{thm}
The deduction of Theorems~\ref{intro:monochromatic} and~\ref{intro:density} from this theorem appears in the appendix as the case $a=b=1$ of Corollary~\ref{app:scaled}, after relabeling the variables.

The proof goes via the Furstenberg correspondence. Informally, consider the space $X$ of subsets in $\Q_{>0}$, the positive rationals, and consider the point $A$ in that space. For a \Folner{} sequence $\Phi$ along which $A$ has positive density, we get a sequence of random variables where we take $A$ and shift all the elements by a random element of $\Phi$. Because $X$ is compact by Tychonoff's theorem, compactness of the space of probability measures on $X$ gives a subsequence along which the distributions of these random variables converge to a multiplication invariant measure $\mu$. Thus we have a compact set $X$ with an action of $\Q_{>0}$ under multiplication and a measure $\mu$ which is invariant for that action. Because $A$ had positive density to begin with, we find that if $f$ is the evaluation at $1$ function that $$\int_X f(x) d\mu(x) > 0$$ and in fact the integral is just the density of $A$ with respect to the \Folner{} sequence. Furthermore, the statement that $$\int_X T_a f \cdot T_b f \cdot T_c f d\mu > 0$$ for some $a, b$ and $c$ in $\Q_{>0}$ implies there exists $k \in \N$ (and in fact a positive upper density set of $k$) such that $ka, kb$ and $kc$ are all in $A$. Therefore, we can reduce Theorem~\ref{intro:general} to the following theorem.
\begin{thm}\label{intro:dynamical}
Let $(P_{i,1},P_{i,2},P_{i,3})_{1\leq i\leq r}$ be a collection that is good for multiplicative recurrence, and let $B\subseteq\N^r$ have positive multiplicative upper density along a multiplicative \Folner{} sequence $\Psi=(\Psi_L)$. Let $(T_q)_{q\in\Q_{>0}}$ be a measure-preserving multiplicative action on a probability space $(X,\mu)$, and let $E\subseteq X$ be measurable with $\mu(E)>0$. Write $F=\ch_E$ and $\beta=\overline d_\Psi(B)$ for the upper density of $B$ along $\Psi$.

For the auxiliary parameters $\mathbf q_i$ used below, let $Q_i$ be the modulus chosen in the proof and $v_i(\mathbf q_i),w_i(\mathbf q_i)$ the residues supplied by the congruence condition, and write
\[
n_i=Q_i n_i'+v_i(\mathbf q_i)
\text{ \ \ and \ \ }
m_i=Q_i m_i'+w_i(\mathbf q_i).
\]
For $r=1$ we omit the index and write $n',m'$. Put $U_{i,j}=P_{i,j}(n_i,m_i)$. The average below is normalized successively over $\ell\in\Psi_L$, the auxiliary parameters $\mathbf q$ used in the proof, and independent pairs $(n_i',m_i')\in[N]^2\cap S_{i,\ell_i,\delta}$. For every $\eta>0$ the sectors can be chosen so that, for all sufficiently small $\delta>0$,
\[
\varlimsup_{L\to\infty}\!\varlimsup_{Q_1\to\infty}
\!\varlimsup_{Q_2\to\infty}\!\varlimsup_{Q_3\to\infty}
\!\varlimsup_{N\to\infty}
\cesE_{\ell,\mathbf q,(n_i',m_i')_i}\ch_B(\ell)
\int_X F\prod_{i=1}^r
\bigl[T_{U_{i,2}/U_{i,3}}F\,T_{U_{i,1}/(\ell_i U_{i,3})}F\bigr]\,d\mu
\geq\beta\mu(E)^{2r+1}-\eta.
\]
In particular, some $\ell\in B$ and $n_i,m_i\in\N$ give a positive integral.
\end{thm}
See \cite{Bergelson_multiplicative,Bergelson-Ferre-correspondence} for more details on the Furstenberg correspondence. Before we get to the proof, let’s discuss how this paper fits into the substantial history of related work.

\subsection{Related results}

For homogeneous linear equations, partition regularity is
characterized by Rado's theorem \cite{Rado-original}. There are
also extensions to diagonal equations in powers. Chow, Lindqvist,
and Prendiville \cite{chow-lindqvist-prendiville} prove that the
zero-subsum criterion characterizes partition regularity of
diagonal quadratic equations with at least five variables, and
obtain corresponding results for higher powers with sufficiently
many variables. Chapman \cite{chapman-part-reg} treats systems
of diagonal equations under a nonsingularity hypothesis. Related
density results for linear equations in the squares were obtained
by Browning and Prendiville \cite{browning-prendiville-roth}.

The extra product in our equation also connects it to questions
involving addition and multiplication together. Moreira
\cite{moreira-monochromatic} proves that every finite coloring
contains a monochromatic configuration
\[
\{x,x+y,xy\}.
\]
For general sufficient and necessary conditions for nonlinear
partition regularity, see Di Nasso and Luperi Baglini
\cite{dinasso-luperi}.

Shalom \cite[Theorems~5.4--5.5]{Shalom} proves a multiple
recurrence theorem for products of linear forms and applies it
to simultaneous solutions of quadratic equations. For example,
every finite coloring of $\N$ admits $x,y,x',y',z,z',k\in\N$
with $x\ne y$ and $x'\ne y'$ such that
\[
16x^2+9y^2=z^2
\text{ \ \ and \ \ }(x')^2+(y')^2-x'y'=(z')^2,
\]
and $xx'/k$, $xy'/k$, and $yy'/k$ are distinct integers of the
same color. In fact, one can deduce Shalom's theorem from
Theorem~\ref{intro:general}, although we will leave the details to the interested reader.

There has been further progress on the Pythagorean problem for
colorings coming from multiplicative functions. Frantzikinakis
and Mountakis \cite{pret_pyth} prove recurrence for pretentious
multiplicative systems along generalized Pythagorean triples.
For $ax^2+by^2=cz^2$ with $a,b,c\in\N$, their result applies
when $a=c$, $b=c$, or $a+b=c$. In particular, it gives Pythagorean triples on which
finitely many unit-circle-valued pretentious completely multiplicative functions
are simultaneously close to $1$. Azevedo and Moreira
\cite[Theorem~1.2]{Azevedo-Moreira} remove the pretentiousness
assumption in this last statement: if $f_1,\ldots,f_s$ are
unit-circle-valued completely multiplicative functions and
$\eta>0$, then there are $x,y,z\in\N$ with
\[
x^2+y^2=z^2
\text{ \ \ and \ \ }
\max_{1\leq j\leq s}\max_{u\in\{x,y,z\}}|f_j(u)-1|<\eta.
\]
They also give a recurrence formulation for multiplicative
Kronecker systems \cite[Theorem~1.7]{Azevedo-Moreira}.

For recurrence along ratios of linear forms, Donoso, Le, Moreira,
and Sun \cite{Moreira&friends} study additive averages of
multiplicative correlation sequences. Charamaras, Mountakis,
and Tsinas \cite{cha-mou-tsin} obtain related approximation
results for multiplicative functions and recurrence results
for finitely generated multiplicative actions. Leung and
T\'afula \cite{Leung-Tafula} give a complete characterization
of measurable and topological multiplicative recurrence for
ratios of two affine linear forms. Our polynomial ratios require
the additional quadratic concentration estimates of
Frantzikinakis, Klurman, and Moreira \cite{fra-klu-mor-2}.

\subsection{Statement of AI use}
While the majority of the ideas were due to the authors, a preliminary AI referee report (via ChatGPT-5.5 Pro) after the project had initially finished led to the discovery of an error that was related to the definition of the sets 
$S_{\delta,\ell}$. The same error appeared in \cite{Fra-Klu-Mor} and in related literature. A previous version of this article had slightly different sets playing the role of $S_{\delta,\ell}$. After we resolved the errors of Theorem~\ref{intro:monochromatic}, it gave us suggestions towards fixing the errors of Theorem~\ref{intro:general}. In particular, it suggested several modifications of Definition~\ref{def:good-multiplicative-recurrence} and how they are relevant towards Theorem~\ref{intro:general}. We would also like to acknowledge the large amount of writing and editing suggestions it gave. ChatGPT-6 was also used for editing and for some of the ideas for applications in the appendix.

\subsection{Notation and Preliminaries}\label{intro:notation}

We write $\N=\{1,2,\ldots\}$, $[N]=\{1,\ldots,N\}$, and
$\Q_{>0}$ for the positive rational numbers. 
A multiplicative \Folner{} sequence $\Phi = \{\Phi_H\}_{H \in \N}$ is a sequence of finite subsets of $\N$ such that, for any fixed number $k$,
\[
\frac{|k\Phi_H \Delta \Phi_H|}{|\Phi_H|} \rightarrow 0 \text{ as } H \rightarrow \infty.
\]
The upper density of a set $A \subset \N$ with respect to $\Phi$ is by definition
\begin{equation}\label{upperdensity}
d_{\Phi}(A) = \limsup_{H \rightarrow \infty} \frac{|A \cap \Phi_H|}{|\Phi_H|}.
\end{equation}
We will say that $A$ has positive upper Banach density if $d_\Phi(A) > 0$ for some multiplicative \Folner{} sequence $\Phi$.

The letters $p$
and $\chi$ denote a prime and a Dirichlet character, respectively.
We write $\nu_p(q)$ for the exponent of $p$ in $q$, so that
$p^e\Vert q$ means $\nu_p(q)=e$. A function is one-bounded
if its absolute value is at most one.

For a finite nonempty set $\Omega$, averages are normalized:
\[
\cesE_{s\in\Omega}a(s)=\frac{1}{|\Omega|}\sum_{s\in\Omega}a(s).
\]
Multiple averages are taken independently unless a joint
restriction is written underneath the averaging symbol. In
particular, if $S\subseteq\N^2$, we put
\[
d_N(S)=\frac{|S\cap[N]^2|}{N^2}
\text{ \ \ and \ \ }
\cesE_{(n',m')\in[N]^2\cap S}a(n',m')
=\frac{\cesE_{n',m'\in[N]}\ch_S(n',m')a(n',m')}{d_N(S)}
\]
whenever $d_N(S)>0$. Here $\ch_S$ is the indicator of $S$.
We use $O(\cdot)$ and $\ll$ for bounds up to a constant;
subscripts record permitted dependencies. For example,
$o_{N\to\infty;L}(1)$ tends to zero as $N\to\infty$ for each
fixed $L$.

Our multiplicative actions $(T_q)_{q\in\Q_{>0}}$ preserve a
probability measure $\mu$ and satisfy $T_qT_r=T_{qr}$. We use
the same notation for their action on functions. Capital letters
$F,G$ denote functions on the probability space, and
$\|F\|_2$ means $\|F\|_{L^2(\mu)}$. We write
$\mathbb E(F\mid\mathcal A)$ for conditional expectation onto
a factor $\mathcal A$. Lowercase letters $f,g$ denote
multiplicative functions. In particular,
\[
\M=\{f:\N\to\mathbb S^1:f(mn)=f(m)f(n)\text{ for all }m,n\in\N\},
\]
where $\mathbb S^1=\{z\in\C:|z|=1\}$. Each $f\in\M$
extends to $\Q_{>0}$ by $f(a/b)=f(a)\overline{f(b)}$.
The notation $\sigma_F$ denotes the spectral measure of $F$
on $\M$. An eigenfunction $G_j$ has eigenvalue $g_j\in\M$,
so that $T_qG_j=g_j(q)G_j$.

For one-bounded multiplicative functions $f,g$, we use the
pretentious distance
\[
\D(f,g;u,v)=\left(\sum_{u<p\leq v}
\frac{1-\operatorname{Re}(f(p)\overline{g(p)})}{p}\right)^{1/2}
\text{ \ \ and \ \ }\D(f,g)=\D(f,g;1,\infty).
\]
We say that $f$ is pretentious if
$\D(f,n\mapsto\chi(n)n^{it})<\infty$ for some Dirichlet
character $\chi$ and $t\in\R$. We call $f$ aperiodic if
$\cesE_{n\in[N]}f(an+b)\to0$ for every $a\in\N$ and every
integer $b\geq0$.

For integers $2\leq u<v$, our auxiliary averaging sets are
\[
\Phi_u=\left\{\prod_{p\leq u}p^{\alpha_p}:
\alpha_p\in\Z,\ u<\alpha_p\leq2u\right\}
\text{ \ \ and \ \ }
\Phi_{u,v}=\left\{\prod_{u<p\leq v}p^{\alpha_p}:
\alpha_p\in\Z,\ v<\alpha_p\leq2v\right\}.
\]
If $R$ is a binary quadratic form, then $\Phi_{u,v,R}$ is
defined in the same way as $\Phi_{u,v}$, using only primes
for which $R(1,a)\equiv0\pmod p$ has a solution. Thus, once
$u\geq2$, the set $\Phi_{u,v,x^2+y^2}$ uses precisely the
primes $p\equiv1\pmod4$ in that interval. Empty products are
taken to be $1$. The sequence $\Psi=(\Psi_L)$ denotes the
multiplicative \Folner{} sequence used to average the parameter
$\ell\in B$; in the simultaneous statement, it is a sequence
in $\N^r$.

We let $Q_1<Q_2<Q_3/2$ be the averaging cutoffs, while the lowercase variables range over
\[
q_{i,1}\in\Phi_{Q_1},\qquad q_{i,2}\in\Phi_{Q_1,Q_2},\qquad
q_{i,3}\in
\begin{cases}
\Phi_{Q_2,Q_3,R_i},&P_{i,3}\text{ has a quadratic factor }R_i,\\
\Phi_{Q_2,Q_3},&\text{otherwise}.
\end{cases}
\]
We put $Q_i=\prod_{p\leq Q_3}p^{3D_iQ_3}$ and use the convention
\[
n_i=Q_i n_i'+v_i(\mathbf q_i)
\text{ \ \ and \ \ }
m_i=Q_i m_i'+w_i(\mathbf q_i),
\]
where the residues $v_i,w_i$ are chosen in $[0,Q_i)$. The
primed variables are the variables being averaged in $[N]^2$;
the unprimed variables are their values in the arithmetic
progressions. The modulus $Q_i$ stays fixed while the
$q_{i,j}$ vary. We write $\mathbf q$ for all the auxiliary variables.

Finally, $S_{i,\ell_i,\delta}$ denotes a sector in the primed
variables of the form
\[
S_{i,\ell_i,\delta}
=\{(n_i',m_i')\in\N^2:n_i'/m_i'\in J_{i,\ell_i,\delta}\}.
\]
The intervals $J_{i,\ell_i,\delta}$ are chosen after fixing the
finite set of frequencies $t_1,\ldots,t_d$ from the eigenfunction
approximation. For each fixed $\delta$, they lie in a common
compact subinterval of $(\kappa_i,\infty)$ and have lengths bounded
below by a positive constant independent of $\ell_i$.
On $S_{i,\ell_i,\delta}$ we require
\[
\left|\log\frac{P_{i,2}(n_i',m_i')}{P_{i,3}(n_i',m_i')}\right|\leq\delta
\text{ \ and \ }
\max_{1\leq j\leq d}
\left|\left(\frac{P_{i,1}(n_i',m_i')}
{\ell_iP_{i,3}(n_i',m_i')}\right)^{it_j}-1\right|\leq\delta.
\]
The dependence on the frequencies is suppressed in the notation.
The sector calculation in the appendix gives these intervals
for the Pythagorean triple. For a general triple, the intervals
are supplied by the sector condition in
Definition~\ref{def:good-multiplicative-recurrence}.

\section{Proof of main theorem}

\subsection{Reduction to the Pretentious Case} \ \\ 
Consider the expression
\[
\cesE_{\ell\in\Psi_L}\ch_B(\ell)\cesE_{\mathbf q}
\cesE_{\substack{(n_i',m_i')\in[N]^2\cap S_{i,\ell_i,\delta}\\1\leq i\leq r}}
\int_X F\prod_{i=1}^r\left[
T_{\frac{P_{i,2}(n_i,m_i)}{P_{i,3}(n_i,m_i)}}F\cdot
T_{\frac{P_{i,1}(n_i,m_i)}{\ell_iP_{i,3}(n_i,m_i)}}F
\right]d\mu.
\]\[
H_i:=F\prod_{\substack{1\leq s\leq r\\s\ne i}}
\left[T_{\frac{P_{s,2}(n_s,m_s)}{P_{s,3}(n_s,m_s)}}F\cdot
T_{\frac{P_{s,1}(n_s,m_s)}{\ell_sP_{s,3}(n_s,m_s)}}F\right].
\]

By the Jacobs--de Leeuw--Glicksberg decomposition
\cite[Theorem~2.24]{DKE}, we can write
\[
F=F_{\wm}+F_{\cpt},
\]
where $F_{\wm}$ is (multiplicatively) weak mixing and
$F_{\cpt}$ is (multiplicatively) compact.

By the definition of weak mixing, we have that
\[
\lim_{H\to\infty}\cesE_{h\in\Phi_H}
\bigl|\langle T_hF_{\wm},F_{\wm}\rangle_{L^2(\mu)}\bigr|=0
\]
for every multiplicative \Folner{} sequence $(\Phi_H)_{H\in\N}$ in $(\N,\times)$.
\begin{claim}
    The terms involving $F_\wm$ tend to $0$.
\end{claim}
\begin{proof}
By linearity of the integral and the triangle inequality, we have that
\[
\begin{aligned}
&\left|\cesE_{\ell\in\Psi_L}\ch_B(\ell)\cesE_{\mathbf q}
\cesE_{\substack{(n_s',m_s')\in[N]^2\cap S_{s,\ell_s,\delta}\\1\leq s\leq r}}
\int_X H_i\,
T_{\frac{P_{i,2}(n_i,m_i)}{P_{i,3}(n_i,m_i)}}F\cdot
T_{\frac{P_{i,1}(n_i,m_i)}{\ell_iP_{i,3}(n_i,m_i)}}F_{\wm}\,d\mu
\right|\\
&\leq \sum_{\lambda\in\pi_{-i}\Psi_L}
\frac{|\Psi_L(\lambda)|}{|\Psi_L|}\,
\cesE_{\mathbf q}\cesE_{\substack{(n_s',m_s')\in[N]^2\cap S_{s,\lambda_s,\delta}\\s\ne i}}
\cesE_{n_i',m_i'\in[N]}\\
&\qquad\int_X\left|H_iT_{\frac{P_{i,2}(n_i,m_i)}{P_{i,3}(n_i,m_i)}}F\right|
\left|\cesE_{\ell_i\in\Psi_L(\lambda)}
\frac{\ch_B(\ell)\ch_{S_{i,\ell_i,\delta}}(n_i',m_i')}
{d_N(S_{i,\ell_i,\delta})}
T_{\frac{P_{i,1}(n_i,m_i)}{\ell_iP_{i,3}(n_i,m_i)}}F_{\wm}\right|d\mu.
\end{aligned}
\]

Because $F$ is one-bounded, we have that this is at most
\[
\begin{aligned}
&\sum_{\lambda\in\pi_{-i}\Psi_L}\frac{|\Psi_L(\lambda)|}{|\Psi_L|}
\cesE_{\mathbf q}\cesE_{\substack{(n_s',m_s')\in[N]^2\cap S_{s,\lambda_s,\delta}\\s\ne i}}
\cesE_{n_i',m_i'\in[N]}\\
&\qquad\int_X
\left|\cesE_{\ell_i\in\Psi_L(\lambda)}
\frac{\ch_B(\ell)\ch_{S_{i,\ell_i,\delta}}(n_i',m_i')}
{d_N(S_{i,\ell_i,\delta})}
T_{\frac{P_{i,1}(n_i,m_i)}{\ell_iP_{i,3}(n_i,m_i)}}F_{\wm}\right|d\mu.
\end{aligned}
\]

By Cauchy--Schwarz and the van der Corput inequality
\cite[cf.\ Theorem~2.12 and its proof]{Bergelson-Moreira-vdC}, the square of this term is at most
\[
\cesE_{h\in\Phi_H}\cesE_{n_i',m_i'\in[N]}
\left|\cesE_{\ell\in\Psi_L}
\frac{\ch_B(\ell^{(i,h)})\ch_B(\ell)
\ch_{S_{i,\ell_i h,\delta}}(n_i',m_i')
\ch_{S_{i,\ell_i,\delta}}(n_i',m_i')}
{d_N(S_{i,\ell_i h,\delta})d_N(S_{i,\ell_i,\delta})}
\langle T_hF_{\wm},F_{\wm}\rangle_{L^2(\mu)}\right|+o(1),
\]
Here we used invariance of $\mu$ in the inner product, and the error
tends to zero as first $L$ and then $H$ tend to infinity.

By the triangle inequality, this is at most
\[
C_\delta\cesE_{h\in\Phi_H}
\bigl|\langle T_hF_{\wm},F_{\wm}\rangle_{L^2(\mu)}\bigr|+o(1),
\]
where $C_\delta$ is a constant depending on $\delta$.
By the definition of weak mixing, this tends to zero.
\end{proof}
Therefore, we may focus on
\[
\cesE_{\ell\in\Psi_L}\ch_B(\ell)\cesE_{\mathbf q}
\cesE_{\substack{(n_s',m_s')\in[N]^2\cap S_{s,\ell_s,\delta}\\1\leq s\leq r}}
\int_X H_i\,
T_{\frac{P_{i,2}(n_i,m_i)}{P_{i,3}(n_i,m_i)}}F\cdot
T_{\frac{P_{i,1}(n_i,m_i)}{\ell_iP_{i,3}(n_i,m_i)}}F_{\cpt}\,d\mu.
\]
In the $P_{i,1}$ term, $F$ has been replaced by $F_{\cpt}$.

Let $\e>0$. By the definition of the compact factor, there are eigenfunctions
$G_1,\ldots,G_d$ with respective eigenvalues $g_1,\ldots,g_d$, and
$\lambda_1,\ldots,\lambda_d\in\C$ such that
\[
\left\|F_{\cpt}-\sum_{j=1}^d\lambda_jG_j\right\|_{L^2(\mu)}<\e,
\]
where $G_j$ is an eigenfunction with eigenvalue $g_j$, meaning that
\[
T_hG_j=g_j(h)G_j
\]
for every $h\in\Q_{>0}$, and $g_j$ is a completely multiplicative
function taking values on the unit circle.
We may take the $G_j$ to be bounded by truncating their invariant
absolute values.

By Cauchy--Schwarz and invariance of $\mu$, replacing $F_{\cpt}$ by
$\sum_{j=1}^d\lambda_jG_j$ changes the expression by at most $\e$.
By linearity, we therefore obtain
\[
\sum_{j=1}^d\lambda_j\cesE_{\ell\in\Psi_L}\ch_B(\ell)
\cesE_{\mathbf q}
\cesE_{\substack{(n_s',m_s')\in[N]^2\cap S_{s,\ell_s,\delta}\\1\leq s\leq r}}
\int_X H_i\,
T_{\frac{P_{i,2}(n_i,m_i)}{P_{i,3}(n_i,m_i)}}F\cdot
T_{\frac{P_{i,1}(n_i,m_i)}{\ell_iP_{i,3}(n_i,m_i)}}G_j\,d\mu+O(\e).
\]

By invariance of $\mu$ and the definition of an eigenfunction, the integral equals
\[
g_j\!\left(\frac{P_{i,1}(n_i,m_i)}{\ell_i P_{i,3}(n_i,m_i)}\right)
\int_X H_i\cdot T_{\frac{P_{i,2}(n_i,m_i)}{P_{i,3}(n_i,m_i)}}F\cdot G_j\,d\mu.
\]

We next show that all the shifted functions may be replaced by their projections onto the pretentious factor.
Write
\[
F=F_{\pret}+F_{\aper},
\]
where $F_{\pret}$ is the projection of $F$ onto the pretentious factor
and $F_{\aper}$ is orthogonal to that factor
\cite[Proposition~5.2 and Corollary~5.5]{Fra-Hom_Triples}.

\begin{claim}
\[
\sum_{j=1}^d\lambda_j\cesE_{\ell\in\Psi_L}\ch_B(\ell)
\cesE_{\mathbf q}
\cesE_{\substack{(n_s',m_s')\in[N]^2\cap S_{s,\ell_s,\delta}\\1\leq s\leq r}}
\int_X H_i\,
T_{\frac{P_{i,2}(n_i,m_i)}{P_{i,3}(n_i,m_i)}}F_{\aper}\cdot
T_{\frac{P_{i,1}(n_i,m_i)}{\ell_iP_{i,3}(n_i,m_i)}}G_j\,d\mu
=\oh_{N\to\infty;L}(1).
\]
\end{claim}
\begin{proof}
Using the preceding identity and Cauchy--Schwarz, it suffices to show that,
for each fixed $i$, $j$ and $\ell_i$,
\[
\left\|
\cesE_{(n_i',m_i')\in[N]^2\cap S_{i,\ell_i,\delta}}
g_j\!\left(\frac{P_{i,1}(n_i,m_i)}{\ell_i P_{i,3}(n_i,m_i)}\right)
T_{\frac{P_{i,2}(n_i,m_i)}{P_{i,3}(n_i,m_i)}}F_{\aper}
\right\|_{L^2(\mu)}
=\oh_{N\to\infty;\ell_i}(1).
\]

Recall that the spectral theorem
\cite[Section~4.4]{Fra-Hom_Triples} says that, for every $F\in L^2(\mu)$,
there is a finite positive Borel measure $\sigma_F$ on $\M$ such that
\[
\left\|\sum_{a=1}^k c_aT_{r_a}F\right\|_{L^2(\mu)}^2
=\int_{\M}\left|\sum_{a=1}^k c_a f(r_a)\right|^2\,d\sigma_F(f)
\]
for every $k\in\N$, $c_1,\ldots,c_k\in\C$, and
$r_1,\ldots,r_k\in\Q_{>0}$.
Here $\M$ is the space of completely multiplicative functions on
$\Q_{>0}$ taking values on the unit circle.

By the spectral theorem, the square of this norm equals
\[
\int_{\M}
\left|
\cesE_{(n_i',m_i')\in[N]^2\cap S_{i,\ell_i,\delta}}
g_j\!\left(\frac{P_{i,1}(n_i,m_i)}{\ell_i P_{i,3}(n_i,m_i)}\right)
f\!\left(\frac{P_{i,2}(n_i,m_i)}{P_{i,3}(n_i,m_i)}\right)
\right|^2
\,d\sigma_{F_{\aper}}(f),
\]
where $\sigma_{F_{\aper}}$ is the spectral measure of $F_{\aper}$.
By the definition of the aperiodic component, $\sigma_{F_{\aper}}$ is
supported on the aperiodic functions, so we may assume $f$ is aperiodic.

By complete multiplicativity, this equals
\[
\int_{\M}\left|
\overline{g_j(\ell_i)}
\cesE_{(n_i',m_i')\in[N]^2\cap S_{i,\ell_i,\delta}}
g_j(P_{i,1}(n_i,m_i))f(P_{i,2}(n_i,m_i))
\overline{(fg_j)(P_{i,3}(n_i,m_i))}
\right|^2\,d\sigma_{F_{\aper}}(f).
\]

Frantzikinakis--Host \cite[Theorem~9.7]{Fra-Host-structure-multiplicative}
showed that, for $Q(m,n)=m^2+n^2$, nonzero pairwise linearly independent
integer linear forms $L_1,\ldots,L_s$, and one-bounded completely
multiplicative functions $f_0,\ldots,f_s$,
\[
\cesE_{\substack{m,n\in[N]\\m<n}}f_0(Q(m,n))
\prod_{k=1}^s f_k(L_k(m,n))=\oh_{N\to\infty}(1),
\]
as long as at least one of $f_1,\ldots,f_s$ is aperiodic. For arbitrary irreducible binary quadratic forms $R_i$ and the affine progressions used here, we use Frantzikinakis--Klurman--Moreira \cite[Theorem~3.1]{fra-klu-mor-2}.
Write $\widetilde P_{i,3}$ for the product of the linear factors of
$P_{i,3}$. After combining proportional linear factors, the
multiplicity-one factor in the reduced ratio
$P_{i,2}/\widetilde P_{i,3}$ that is independent of the factors
remaining in $P_{i,1}/\widetilde P_{i,3}$ carries $f$ or $\overline f$.
When $P_{i,3}$ has no quadratic factor, we may use
$n^2+m^2$ with the constant function $1$ as the quadratic term
in the cited theorem. Since $f$ is aperiodic, the integrand tends to zero
\cite[Theorem~3.1]{fra-klu-mor-2}.
By the dominated convergence theorem, the integral also tends to zero.

Similarly, we can show that
\[
\sum_{i=1}^r\left|\sum_{j=1}^d\lambda_j
\cesE_{\ell\in\Psi_L}\ch_B(\ell)\cesE_{\mathbf q}
\cesE_{\substack{(n_s',m_s')\in[N]^2\cap S_{s,\ell_s,\delta}\\1\leq s\leq r}}
\int_X H_i\cdot T_{\frac{P_{i,2}(n_i,m_i)}{P_{i,3}(n_i,m_i)}}F_{\aper}
\cdot T_{\frac{P_{i,1}(n_i,m_i)}{\ell_iP_{i,3}(n_i,m_i)}}G_j\,d\mu
\right|=\oh_{N\to\infty;L}(1).
\]

Finally, if $g_j$ is aperiodic, then applying invariance, Cauchy--Schwarz,
and the spectral theorem again, for each fixed $i$ and $\ell_i$,
\begin{align*}
&\left|\cesE_{(n_i',m_i')\in[N]^2\cap S_{i,\ell_i,\delta}}
\int_X H_i\cdot T_{\frac{P_{i,2}(n_i,m_i)}{P_{i,3}(n_i,m_i)}}F_{\pret}
\cdot T_{\frac{P_{i,1}(n_i,m_i)}{\ell_i P_{i,3}(n_i,m_i)}}G_j\,d\mu\right|^2\\
&\leq\|G_j\|_{L^2(\mu)}^2
\left\|\cesE_{(n_i',m_i')\in[N]^2\cap S_{i,\ell_i,\delta}}
g_j\!\left(\frac{P_{i,1}(n_i,m_i)}{\ell_i P_{i,3}(n_i,m_i)}\right)
T_{\frac{P_{i,2}(n_i,m_i)}{P_{i,3}(n_i,m_i)}}F_{\pret}\right\|_{L^2(\mu)}^2\\
&=\|G_j\|_{L^2(\mu)}^2
\int_{\M}\left|\cesE_{(n_i',m_i')\in[N]^2\cap S_{i,\ell_i,\delta}}
g_j(P_{i,1}(n_i,m_i))f(P_{i,2}(n_i,m_i))
\overline{(fg_j)(P_{i,3}(n_i,m_i))}\right|^2\,d\sigma_{F_{\pret}}(f).
\end{align*}
After combining proportional linear factors, the multiplicity-one
factor in the reduced ratio $P_{i,1}/\widetilde P_{i,3}$ carries
$g_j$ or $\overline{g_j}$ multiplied by a pretentious function,
which is still aperiodic.
As above, in the all-linear case we can use a constant quadratic term.
Since $g_j$ is aperiodic, Frantzikinakis--Host \cite[Theorem~9.7]{Fra-Host-structure-multiplicative}, or Frantzikinakis--Klurman--Moreira \cite[Theorem~3.1]{fra-klu-mor-2} for a general quadratic form, again implies that the inner
average tends to zero, which implies that the integral tends to zero.
Consequently,
\[
\cesE_{\ell\in\Psi_L}\ch_B(\ell)\cesE_{\mathbf q}
\cesE_{\substack{(n_s',m_s')\in[N]^2\cap S_{s,\ell_s,\delta}\\1\leq s\leq r}}
\int_X H_i\cdot T_{\frac{P_{i,2}(n_i,m_i)}{P_{i,3}(n_i,m_i)}}F_{\pret}
\cdot T_{\frac{P_{i,1}(n_i,m_i)}{\ell_iP_{i,3}(n_i,m_i)}}G_j\,d\mu
=\oh_{N\to\infty;L}(1).
\]
Thus we may restrict the sum to those $j$ for which $g_j$ is pretentious.
\end{proof}

\subsection{Independence of the Averages in $q_i$}
\ \\ Therefore, we are reduced to showing that
\[
\begin{aligned}
&\limsup_{L\to\infty}\limsup_{N\to\infty}
\cesE_{\ell\in\Psi_L}\ch_B(\ell)\cesE_{\mathbf q}\\
&\qquad
\cesE_{\substack{(n_i',m_i')\in[N]^2\cap S_{i,\ell_i,\delta}\\1\leq i\leq r}}
\int_X F\prod_{i=1}^r\left[
T_{\frac{P_{i,2}(n_i,m_i)}{P_{i,3}(n_i,m_i)}}F_{\pret}\cdot
T_{\frac{P_{i,1}(n_i,m_i)}{\ell_iP_{i,3}(n_i,m_i)}}
\E(F\mid X_{\cpt}\cap X_{\pret})
\right]d\mu\ge c>0.
\end{aligned}
\]
where $c$ is independent of $\e$.

By the linear concentration estimate
\cite[Lemma~2.5]{Klu-Man-Po-Ter}, in the formulation of
\cite[Proposition~2.1]{pret_pyth}, if $q,a,y\in\N$ satisfy
$(a,q)=1$, the conductor of $\chi$ divides $q$, and the prime divisors
of $q$ are precisely the primes at most $y$, then
\[
\limsup_{N\to\infty}\cesE_{n\in[N]}
\left|f(qn+a)-\chi(a)(qn)^{it}\exp\bigl(F_N(f,y)\bigr)\right|
\ll \D\bigl(f,n\mapsto\chi(n)n^{it};y,\infty\bigr)+y^{-1/2},
\]
where
\[
F_N(f,y):=\sum_{y<p\le N}\frac{h_f(p)-1}{p}
\;\;\text{ and }\;\;
h_f(n):=f(n)\overline{\chi(n)}n^{-it}.
\]

Remember that $n_i$ and $m_i$ are defined by
\[
n_i=Q_i n_i'+v_i(q_{i,1},q_{i,2},q_{i,3})\text{ \ and \ }m_i=Q_i m_i'+w_i(q_{i,1},q_{i,2},q_{i,3}),
\]
so $n_i$ and $m_i$ implicitly depend on $q_{i,1},q_{i,2},q_{i,3}$ through $v_i$ and $w_i$.

We will use the concentration estimates to show that the $T_{P_{i,2}(n_i,m_i)}F_{\pret}$ term essentially only depends on $q_{i,2}$, and the $T_{P_{i,1}(n_i,m_i)/\ell_i}G_j$ term essentially only depends on $q_{i,1}$ \cite[Lemma~4.4]{pret_pyth}.

Let's start with the first term. Fix $q_{i,2}$.
For almost every $f$ in the support of $\sigma_{F_{\pret}}$, choose
a Dirichlet character $\chi$ and $t\in\R$ such that
$\D(f,n\mapsto\chi(n)n^{it})<\infty$, which we can do by the
definition of the pretentious factor.
By the choice of $v_i$, we can write
\[
P_{i,2}(n_i,m_i)=\prod_{k=1}^{D_i}L_{i,2,k}(n_i,m_i)
=\prod_{k=1}^{D_i}\bigl(Q_i L_{i,2,k}(n_i',m_i')+L_{i,2,k}(v_i,w_i)\bigr).
\]
For each linear factor, let
\[
c_{i,k}=\gcd\bigl(Q_i,L_{i,2,k}(v_i,w_i)\bigr),
\]
and apply the concentration estimate with
\[
(q,a,y)=\left(\frac{Q_i}{c_{i,k}},\frac{L_{i,2,k}(v_i,w_i)}{c_{i,k}},Q_3\right),
\]
and variable $L_{i,2,k}(n_i',m_i')$. The residue construction gives
$\gcd(P_{i,2}(v_i,w_i),Q_i)=q_{i,2}$, so the reduced residues are
coprime to their moduli. These moduli have precisely the primes at
most $Q_3$ as prime divisors, and, for $Q_1$ sufficiently large,
they are divisible by the conductor of $\chi$.
By the linear concentration estimate, we have
\[
f\!\left(\frac{L_{i,2,k}(n_i,m_i)}{c_{i,k}}\right)
\approx\chi\!\left(\frac{L_{i,2,k}(v_i,w_i)}{c_{i,k}}\right)
\left(\frac{Q_i L_{i,2,k}(n_i',m_i')}{c_{i,k}}\right)^{it}e^{F_N(f,Q_3)}.
\]
Here $\approx$ means that, for every fixed error tolerance, all these approximations hold within that tolerance for a proportion of pairs
$(n_i',m_i')\in[N]^2\cap S_{i,\ell_i,\delta}$ tending to one as first $N$
and then $Q_3$ tend to infinity, with $f,\delta,Q_1,Q_2$ fixed.
By Markov's inequality, the proportion of exceptional pairs for each
approximation is bounded by its mean error divided by the tolerance.
These mean errors tend to zero because the pretentious distance is
finite, and the union of the finitely many exceptional sets also has proportion
tending to zero.

Since the period of $\chi$ divides $q_{i,1}$, the congruence $P_{i,2}(v_i,w_i)\equiv1\pmod{q_{i,1}}$ gives
\[
\chi\!\left(\frac{P_{i,2}(v_i,w_i)}{q_{i,2}}\right)=\overline{\chi(q_{i,2})}.
\]
By complete multiplicativity and the triangle inequality, multiplying
these approximations gives
\[
f(P_{i,2}(n_i,m_i))\approx h_f(q_{i,2})
\bigl(Q_i^{D_i}P_{i,2}(n_i',m_i')\bigr)^{it}e^{D_iF_N(f,Q_3)},
\]
where the expression on the right does not depend on $q_{i,1}$ or $q_{i,3}$.

Similarly, the second term, $T_{P_{i,1}(n_i,m_i)/\ell_i}G_j$, essentially only depends on $q_{i,1}$. After all, setting
\[
c_{i,k}=\gcd\bigl(Q_i,L_{i,1,k}(v_i,w_i)\bigr),
\]
with variable $L_{i,1,k}(n_i',m_i')$ gives
\[
f\!\left(\frac{L_{i,1,k}(n_i,m_i)}{c_{i,k}}\right)
\approx\chi\!\left(\frac{L_{i,1,k}(v_i,w_i)}{c_{i,k}}\right)
\left(\frac{Q_i L_{i,1,k}(n_i',m_i')}{c_{i,k}}\right)^{it}e^{F_N(f,Q_3)}.
\]
where we have used $P_{i,1}(v_i,w_i)/q_{i,1}\equiv1\pmod{q_{i,1}}$ and taken $Q_1$
sufficiently large that the period of $\chi$ divides $q_{i,1}$.
Again, the union of the finitely many exceptional sets has proportion tending
to zero. By complete multiplicativity and the triangle inequality,
\[
f(P_{i,1}(n_i,m_i))\approx f(q_{i,1})q_{i,1}^{-it}
\bigl(Q_i^{D_i}P_{i,1}(n_i',m_i')\bigr)^{it}e^{D_iF_N(f,Q_3)},
\]
which does not depend on $q_{i,2}$ or $q_{i,3}$. Taking $f=g_j$ and using
\[
T_{P_{i,1}(n_i,m_i)/\ell_i}G_j=\overline{g_j(\ell_i)}g_j(P_{i,1}(n_i,m_i))G_j
\]
gives the corresponding approximation for the second term.

For the third term, when $P_{i,3}$ has a quadratic factor $R_i$, we need the following quadratic concentration
estimate \cite[Propositions~4.8--4.9]{fra-klu-mor-2}, in the
formulation of \cite[Proposition~2.2]{pret_pyth}. Let $f,\chi,t$
be as above, and put $c=\gcd(q,R_i(a,b))$.
For $y$ sufficiently large, suppose that all prime divisors of $q$
are at most $y$, and that $q/c$ is divisible by the conductor of
$\chi$ and by every prime at most $y$. Then
\[
\begin{aligned}
&\limsup_{N\to\infty}\cesE_{u,w\in[N]} 1_{R_i(qu+a,qw+b)>0}
\left|f\!\left(\frac{R_i(qu+a,qw+b)}c\right)
-\chi\!\left(\frac{R_i(a,b)}c\right)
\left(\frac{q^2R_i(u,w)}c\right)^{it}e^{G_{R_i,N}(f,y)}\right|\\
&\hspace{8em}\ll\D\bigl(f,n\mapsto\chi(n)n^{it};y,\infty\bigr)+y^{-1/2}.
\end{aligned}
\]
where
\[
G_{R_i,N}(f,y):=\sum_{y<p\leq N}\frac{\omega_{R_i}(p)(h_f(p)-1)}p,
\]

Fix $q_{i,3}$. Setting
\[
(q,a,b,y)=(Q_i,v_i,w_i,Q_3)\text{ \ and \ }c=\gcd(Q_i,R_i(v_i,w_i)).
\]
gives
\[
f\!\left(\frac{R_i(n_i,m_i)}c\right)
\approx\chi\!\left(\frac{R_i(v_i,w_i)}c\right)
\left(\frac{Q_i^2R_i(n_i',m_i')}c\right)^{it}e^{G_{R_i,N}(f,Q_3)}.
\]
Since $P_{i,3}(v_i,w_i)\equiv1\pmod{q_{i,1}}$ and the period of $\chi$ divides $q_{i,1}$, the product of the character factors is $\overline{\chi(q_{i,3})}$. By complete
multiplicativity,
\[
f(P_{i,3}(n_i,m_i))\approx h_f(q_{i,3})
\bigl(Q_i^{D_i}P_{i,3}(n_i',m_i')\bigr)^{it}
 e^{s_iF_N(f,Q_3)+\gamma_i G_{R_i,N}(f,Q_3)},\text{ where } s_i+2\gamma_i=D_i,
\]
where the right-hand side does not depend on $q_{i,1}$ or $q_{i,2}$.

Implicitly going forward the we can write expressions like 
\[
T_{\frac{P_{i,2}(n_i,m_i)}{P_{i,3}(n_i,m_i)}}F_{\pret}
\]
and understand that the values of $n_i$ and $m_i$ which usually might depend on $q_{i,1}$ and $q_{i,3}$ in fact do not depend on those choices. Instead, we can simply have $n_i$ and $m_i$ in this expression depend on a sufficiently general but fixed choice of those two parameters.

Therefore,
\[
\begin{aligned}
&\cesE_{q_{i,1}\in\Phi_{Q_1}}\cesE_{q_{i,2}\in\Phi_{Q_1,Q_2}}
\cesE_{q_{i,3}}
\cesE_{(n_i',m_i')\in[N]^2\cap S_{i,\ell_i,\delta}}
T_{\frac{P_{i,2}(n_i,m_i)}{P_{i,3}(n_i,m_i)}}F_{\pret}\cdot
T_{\frac{P_{i,1}(n_i,m_i)}{\ell_i P_{i,3}(n_i,m_i)}}G_j\\
&=\cesE_{(n_i',m_i')\in[N]^2\cap S_{i,\ell_i,\delta}}
\cesE_{q_{i,3}}T_{1/P_{i,3}(n_i,m_i)}\left(
\left[\cesE_{q_{i,2}\in\Phi_{Q_1,Q_2}}T_{P_{i,2}(n_i,m_i)}F_{\pret}\right]
\left[\cesE_{q_{i,1}\in\Phi_{Q_1}}T_{P_{i,1}(n_i,m_i)/\ell_i}G_j\right]
\right)+o_{L^2}(1).
\end{aligned}
\]
where we interpret the first $n_i,m_i$ as depending on $q_{i,2}$, for an
arbitrary fixed choice of $q_{i,1}$ and $q_{i,3}$, the second $n_i,m_i$ as depending
on $q_{i,1}$, for an arbitrary fixed choice of $q_{i,2}$ and $q_{i,3}$, and the
third $n_i,m_i$ as depending on $q_{i,3}$, for an arbitrary fixed choice of
$q_{i,1}$ and $q_{i,2}$.

\subsection{Reduction to the Exact Factor}
\ \\
For the first term, we write $F_{\pret}=F_{\ext}+F_{\apx}$, where
$\sigma_{F_{\ext}}$ is supported on multiplicative functions exactly agreeing
with $\chi(p)p^{it}$ for all but finitely many primes, for some Dirichlet
character $\chi$ and some $t\in\R$:
\begin{equation}\label{eq:exact-agreement}
f(p)=\chi(p)p^{it}
\text{ for all but finitely many primes }p.
\end{equation}
The measure $\sigma_{F_{\apx}}$ is supported on pretentious multiplicative
functions for which no such exact agreement holds.
More precisely, we define \cite[Lemmas~2.3 and~2.5]{pret_pyth}
\[
\begin{gathered}
F_{\ext}:=\mathbb{E}(F_{\pret}\mid X_{p,\exact}),\\
F_{\apx}:=F_{\pret}-F_{\ext}.
\end{gathered}
\]
\begin{claim} 
The contribution from terms involving $F_{\apx}$ is negligible.
\end{claim}
\begin{proof}
Fix $Q_1<Q_2<Q_3/2$ and put
\[
Q_i:=\prod_{p\leq Q_3}p^{3D_iQ_3}.
\]
Fix $q_{i,1}$ and $q_{i,3}$ in their respective averaging sets. For each
$q_{i,2}\in\Phi_{Q_1,Q_2}$, write $v_i=v_i(q_{i,1},q_{i,2},q_{i,3})$ and $w_i=w_i(q_{i,1},q_{i,2},q_{i,3})$ and let
\[
\begin{gathered}
n_i=Q_i n_i'+v_i,\\m_i=Q_i m_i'+w_i.
\end{gathered}
\]
As $q_{i,2}$ varies, $Q_i$ stays fixed and $v_i,w_i$ may change.

We will prove that, for each fixed $\delta$ and $\ell_i$,
\[
\limsup_{Q_1\to\infty}\limsup_{Q_2\to\infty}
\limsup_{Q_3\to\infty}\limsup_{N\to\infty}
\cesE_{(n_i',m_i')\in[N]^2\cap S_{i,\ell_i,\delta}}
\left\|
\cesE_{q_{i,2}\in\Phi_{Q_1,Q_2}}
T_{P_{i,2}(n_i,m_i)}F_{\apx}
\right\|_{L^2(\mu)}^2=0.
\]

By the spectral theorem,
\begin{align*}
&\cesE_{(n_i',m_i')\in[N]^2\cap S_{i,\ell_i,\delta}}
\left\|
\cesE_{q_{i,2}\in\Phi_{Q_1,Q_2}}
T_{P_{i,2}(n_i,m_i)}F_{\apx}
\right\|_{L^2(\mu)}^2\\
&=\int_{\M}
\cesE_{(n_i',m_i')\in[N]^2\cap S_{i,\ell_i,\delta}}
\left|
\cesE_{q_{i,2}\in\Phi_{Q_1,Q_2}}
f\big(P_{i,2}(n_i,m_i)\big)
\right|^2\,d\sigma_{F_{\apx}}(f).
\end{align*}
Since $\sigma_{F_{\apx}}$ is supported on pretentious functions, for
$\sigma_{F_{\apx}}$-almost every $f$ we can find a Dirichlet character
$\chi$ and some $t\in\R$ such that
\[
\D\bigl(f,n_i\mapsto\chi(n_i)n_i^{it}\bigr)<\infty.
\]
However, by the definition of $F_{\apx}$, \eqref{eq:exact-agreement}
does not hold for such $f$. In particular,
\[
f(p)\ne\chi(p)p^{it}
\text{ for infinitely many primes }p.
\]
Write
\[
h_f(n_i):=f(n_i)\overline{\chi(n_i)}n_i^{-it}.
\]

By the choice of $v_i,w_i$, we have
\begin{equation}\label{eq:concentration-residue-conditions}
\begin{gathered}
\gcd(P_{i,j}(v_i,w_i),Q_i)=q_{i,j}\quad(j=1,2,3),\\
P_{i,1}(v_i,w_i)\equiv q_{i,1}\pmod{q_{i,1}^2},\qquad
P_{i,2}(v_i,w_i)\equiv P_{i,3}(v_i,w_i)\equiv1\pmod{q_{i,1}}.
\end{gathered}
\end{equation}
In particular, we can factor $P_{i,2}$ as
\begin{equation}\label{eq:concentration-linear-factors}
P_{i,2}(n_i,m_i)=\prod_{k=1}^{D_i}L_{i,2,k}(n_i,m_i).
\end{equation}
For each linear factor, we apply concentration with
\[
c_{i,k}=\gcd\bigl(Q_i,L_{i,2,k}(v_i,w_i)\bigr),
\]
and variable $L_{i,2,k}(n_i',m_i')$. For the next factor, we use
\[
(q,a,y)=\left(\frac{Q_i}{c_{i,k}},\frac{L_{i,2,k}(v_i,w_i)}{c_{i,k}},Q_3\right),
\]
and variable $L_{i,2,k}(n_i',m_i')$. The gcd conditions in
\eqref{eq:concentration-residue-conditions} ensure that each reduced residue
is coprime to its modulus. Also, dividing $Q_i$ by $c_{i,k}$ leaves every
prime up to $Q_3$ in the modulus: each exponent in $Q_i$ is $3D_iQ_3$,
whereas $c_{i,k}$ removes at most the exponent of $q_{i,2}$. Finally, for $Q_1$ sufficiently
large, the period of $\chi$ divides $q_{i,1}$, and therefore divides each $Q_i/c_{i,k}$.

By the linear concentration estimate,
\[
f\!\left(\frac{L_{i,2,k}(n_i,m_i)}{c_{i,k}}\right)
\approx\chi\!\left(\frac{L_{i,2,k}(v_i,w_i)}{c_{i,k}}\right)
\left(\frac{Q_i L_{i,2,k}(n_i',m_i')}{c_{i,k}}\right)^{it}e^{F_N(f,Q_3)}.
\]
Recall that $\approx$ means that, for every fixed error tolerance, all these
approximations hold within that tolerance for a proportion of pairs
$(n_i',m_i')\in[N]^2\cap S_{i,\ell_i,\delta}$ tending to one as first $N$ and
then $Q_3$ tend to infinity, with $f,\delta,Q_1,Q_2$ fixed.
Since
$\Phi_{Q_1,Q_2}$ is fixed and finite, we may discard the union of the
exceptional sets and have all the approximations hold for every $q_{i,2}$
simultaneously.

Since the period of $\chi$ divides $q_{i,1}$, the congruences in
\eqref{eq:concentration-residue-conditions} give
\[
\chi\!\left(\frac{P_{i,2}(v_i,w_i)}{q_{i,2}}\right)=\overline{\chi(q_{i,2})}.
\]
By complete multiplicativity and the triangle inequality, multiplying
the approximations and averaging over $q_{i,2}$ gives
\[
\cesE_{q_{i,2}\in\Phi_{Q_1,Q_2}}f(P_{i,2}(n_i,m_i))
\approx\bigl(Q_i^{D_i}P_{i,2}(n_i',m_i')\bigr)^{it}e^{D_iF_N(f,Q_3)}
\cesE_{q_{i,2}\in\Phi_{Q_1,Q_2}}h_f(q_{i,2}).
\]
The factor outside the average is independent of $q_{i,2}$ and has modulus
at most one, since $\operatorname{Re}(h_f(p)-1)\le0$ for $p>Q_3$.
By the triangle inequality and Cauchy--Schwarz,

\[
\limsup_{N\to\infty}
\cesE_{(n_i',m_i')\in[N]^2\cap S_{i,\ell_i,\delta}}
\left|\cesE_{q_{i,2}\in\Phi_{Q_1,Q_2}}f(P_{i,2}(n_i,m_i))\right|^2
\le
\left|\cesE_{q_{i,2}\in\Phi_{Q_1,Q_2}}h_f(q_{i,2})\right|^2
+\oh_{Q_3\to\infty}(1).
\]
Here the error may depend on $f,\delta,Q_1,Q_2$.

Let $p>Q_1$ be a prime for which $h_f(p)\ne1$, which exists because
\eqref{eq:exact-agreement} does not hold. For $Q_2$ sufficiently large,
shifting the average by $p$ gives
\[
\cesE_{q_{i,2}\in\Phi_{Q_1,Q_2}}h_f(q_{i,2})
=h_f(p)\cesE_{q_{i,2}\in\Phi_{Q_1,Q_2}}h_f(q_{i,2})+o_{Q_2\to\infty}(1),
\]
so this average tends to zero. Therefore, the contribution from
$F_{\apx}$ is negligible.

We now turn to the second bracket. Choose a Dirichlet character $\chi_j$
and $t_j\in\R$ such that
\[
\D(g_j,n_i\mapsto\chi_j(n_i)n_i^{it_j})<\infty.
\]
For $Q_1$ sufficiently large, the period of $\chi_j$ divides $q_{i,1}$.
Since $P_{i,1}(v_i,w_i)/q_{i,1}\equiv1\pmod{q_{i,1}}$, we have
\[
\chi_j(1)=\chi_j\!\left(\frac{P_{i,1}(v_i,w_i)}{q_{i,1}}\right)=1.
\]
Applying the linear concentration estimate to each factor $L_{i,1,k}(n_i,m_i)$,
and using complete multiplicativity, therefore gives
\[
g_j(P_{i,1}(n_i,m_i))\approx
\bigl(Q_i^{D_i}P_{i,1}(n_i',m_i')\bigr)^{it_j}e^{D_iF_N(g_j,Q_3)}
 g_j(q_{i,1})q_{i,1}^{-it_j}.
\]
The factor outside $g_j(q_{i,1})q_{i,1}^{-it_j}$ is independent of $q_{i,1}$
and has modulus at most one. By the definition of an eigenfunction
and the triangle inequality,
\begin{align*}
&\limsup_{N\to\infty}
\cesE_{(n_i',m_i')\in[N]^2\cap S_{i,\ell_i,\delta}}
\left\|\cesE_{q_{i,1}\in\Phi_{Q_1}}T_{P_{i,1}(n_i,m_i)/\ell_i}G_j\right\|_2^2\\
&\le\|G_j\|_2^2
\left|\cesE_{q_{i,1}\in\Phi_{Q_1}}g_j(q_{i,1})q_{i,1}^{-it_j}\right|^2
+\oh_{Q_3\to\infty}(1).
\end{align*}

If $g_j$ is not identically $n\mapsto n^{it_j}$, then there exists
a prime $p$ for which $g_j(p)p^{-it_j}\ne1$. Shifting the average
by $p$, as before, gives
\begin{align*}
\cesE_{q_{i,1}\in\Phi_{Q_1}}g_j(q_{i,1})q_{i,1}^{-it_j}
&=g_j(p)p^{-it_j}
\cesE_{q_{i,1}\in\Phi_{Q_1}}g_j(q_{i,1})q_{i,1}^{-it_j}\\
&\phantom{{}={}}+\oh_{Q_1\to\infty}(1),
\end{align*}
so this average tends to zero. Thus we may restrict the sum to
those $j$ for which $g_j(n)=n^{it_j}$.
The closed span of all eigenfunctions with eigenvalues of the form
$n\mapsto n^{it}$ is the compact Archimedean factor
\cite[Lemma~2.5]{pret_pyth}.
Therefore, we may replace the eigenfunction sum by its projection onto this
factor, keeping only the eigenfunctions that are measurable with
respect to it.

Finally, the same argument applies to the third bracket, except that if $P_{i,3}$ has a quadratic factor $R_i$, $q_{i,3}\in\Phi_{Q_2,Q_3,R_i}$ only involves primes for which $R_i(1,a)\equiv0\pmod p$ has a solution. Otherwise $q_{i,3}\in\Phi_{Q_2,Q_3}$ uses all primes in that interval.
For a pretentious $f$, choose $\chi,t$ as before. If
$f(p)\ne\chi(p)p^{it}$ for infinitely many such primes, then we can
choose a prime $p>Q_2$ for which $h_f(p)\ne1$. Shifting the average
by $p$ again gives
\begin{align*}
\cesE_{q_{i,3}}h_f(q_{i,3})
&=h_f(p)\cesE_{q_{i,3}}h_f(q_{i,3})
+\oh_{Q_3\to\infty}(1),
\end{align*}
so this average tends to zero. By the quadratic concentration
estimate and the spectral theorem, these components again make a
negligible contribution. Thus, in the third bracket, we may project the product of the
first two terms onto the factor whose spectral measures are supported
on functions satisfying
\[
\begin{gathered}
\mathcal P_i:=
\begin{cases}
\{p:R_i(1,a)\equiv0\pmod p\text{ for some }a\},&P_{i,3}\text{ has }R_i,\\
\{p:p\text{ prime}\},&\text{otherwise},
\end{cases}\\
f(p)=\chi(p)p^{it}\text{ for all but finitely many }p\in\mathcal P_i,
\end{gathered}
\]
for some Dirichlet character $\chi$ and some $t\in\R$.
\end{proof}

\subsection{Finishing the Argument}
\ \\
Write $F_1,F_2$ for the conditional expectations obtained in the
first and second terms, respectively. In particular, $F_1=F_{\ext}$
and $F_2$ is the conditional expectation of $F$ onto the compact
Archimedean factor. Since both are measurable with respect to the
factor just described, the third projection leaves their product
unchanged.

By the previous arguments, we are reduced to the expression
\[
\begin{aligned}
&\cesE_{\ell\in\Psi_L}\ch_B(\ell)
\cesE_{\substack{(n_i',m_i')\in[N]^2\cap S_{i,\ell_i,\delta}\\1\leq i\leq r}}
\int_X F\prod_{i=1}^r\\
&\qquad\cesE_{q_{i,3}}T_{1/P_{i,3}(n_i,m_i)}\Bigl(
\left[\cesE_{q_{i,2}}T_{P_{i,2}(n_i,m_i)}F_1\right]
\left[\cesE_{q_{i,1}}T_{P_{i,1}(n_i,m_i)/\ell_i}F_2\right]
\Bigr)\,d\mu.
\end{aligned}
\]
\begin{claim}
The previous expression is positive.
\end{claim}
\begin{proof}
Because the first term essentially does not depend on $q_{i,1},q_{i,3}$,
the second on $q_{i,2},q_{i,3}$, we can pull the averages outside and obtain, for each fixed $\ell_i$,
\[
\begin{aligned}
&\cesE_{\substack{(n_i',m_i')\in[N]^2\cap S_{i,\ell_i,\delta}\\1\leq i\leq r}}
\int_X F\prod_{i=1}^r\cesE_{q_{i,3}}T_{1/P_{i,3}(n_i,m_i)}\Bigl(\\
&\qquad\left[\cesE_{q_{i,2}}T_{P_{i,2}(n_i,m_i)}F_1\right]
\left[\cesE_{q_{i,1}}T_{P_{i,1}(n_i,m_i)/\ell_i}F_2\right]\Bigr)d\mu\\
&=\cesE_{\mathbf q}
\cesE_{\substack{(n_i',m_i')\in[N]^2\cap S_{i,\ell_i,\delta}\\1\leq i\leq r}}
\int_X F\prod_{i=1}^r\left[
T_{\frac{P_{i,2}(n_i,m_i)}{P_{i,3}(n_i,m_i)}}F_1\cdot
T_{\frac{P_{i,1}(n_i,m_i)}{\ell_iP_{i,3}(n_i,m_i)}}F_2
\right]d\mu+o(1).
\end{aligned}
\]
Here the $n_i$'s and $m_i$'s on the left are interpreted as before, while on the
right we use $n_i=Q_i n_i'+v_i(\mathbf q_i)$ and $m_i=Q_i m_i'+w_i(\mathbf q_i)$ in both numerator terms and the denominator shift.
By invariance of $\mu$,
\begin{align*}
&\int_X T_{P_{i,2}(n_i,m_i)}F_1\cdot T_{P_{i,1}(n_i,m_i)/\ell_i}F_2
\cdot T_{P_{i,3}(n_i,m_i)}H_i\,d\mu\\
&=\int_X H_i\cdot T_{\frac{P_{i,2}(n_i,m_i)}{P_{i,3}(n_i,m_i)}}F_1
\cdot T_{\frac{P_{i,1}(n_i,m_i)}{\ell_i P_{i,3}(n_i,m_i)}}F_2\,d\mu.
\end{align*}
We first show that we can replace the term $T_{\frac{P_{i,2}(n_i,m_i)}{P_{i,3}(n_i,m_i)}}F_1$ by $F_1$.
To do this, consider
\[
\cesE_{q_{i,1}}\cesE_{q_{i,2}}\cesE_{q_{i,3}}
\cesE_{(n_i',m_i')\in[N]^2\cap S_{i,\ell_i,\delta}}
\int_X H_i\cdot
\left[T_{\frac{P_{i,2}(n_i,m_i)}{P_{i,3}(n_i,m_i)}}F_1-F_1\right]
T_{\frac{P_{i,1}(n_i,m_i)}{\ell_i P_{i,3}(n_i,m_i)}}F_2\,d\mu.
\]
where the $\mathbf q_i$-averages are over the same sets as above.
Since $H_i$ and $F_2$ are one-bounded, Cauchy--Schwarz bounds the square
of the absolute value of this expression by
\[
\cesE_{q_{i,1}}\cesE_{q_{i,2}}\cesE_{q_{i,3}}
\cesE_{(n_i',m_i')\in[N]^2\cap S_{i,\ell_i,\delta}}
\left\|T_{\frac{P_{i,2}(n_i,m_i)}{P_{i,3}(n_i,m_i)}}F_1-F_1\right\|_2^2.
\]
By the spectral theorem, this equals
\begin{equation}\label{eq:first-shift-spectral-integral}
\int_{\M}
\cesE_{q_{i,1}}\cesE_{q_{i,2}}\cesE_{q_{i,3}}
\cesE_{(n_i',m_i')\in[N]^2\cap S_{i,\ell_i,\delta}}
\left|f\!\left(\frac{P_{i,2}(n_i,m_i)}{P_{i,3}(n_i,m_i)}\right)-1\right|^2
\,d\sigma_{F_1}(f).
\end{equation}

By the definition of $F_1$, for $\sigma_{F_1}$-almost every $f$
there are $\chi,t$ such that $f(p)=\chi(p)p^{it}$ for all but
finitely many primes $p$. Once $Q_1$ is sufficiently large, we have
\[
h_f(q_{i,2})=h_f(q_{i,3})=1\text{ \ and \ }
F_N(f,Q_3)=\gamma_i G_{R_i,N}(f,Q_3)=0.
\]
The concentration estimates used above therefore give
\[
\begin{aligned}
f(P_{i,2}(n_i,m_i))&\approx\bigl(Q_i^{D_i}P_{i,2}(n_i',m_i')\bigr)^{it},\\
f(P_{i,3}(n_i,m_i))&\approx\bigl(Q_i^{D_i}P_{i,3}(n_i',m_i')\bigr)^{it}.
\end{aligned}
\]
Both approximations hold outside small exceptional sets, whose union
is again small. Since the values of $f$ have modulus one, the triangle
inequality gives
\[
f\!\left(\frac{P_{i,2}(n_i,m_i)}{P_{i,3}(n_i,m_i)}\right)
\approx
\left(\frac{P_{i,2}(n_i',m_i')}{P_{i,3}(n_i',m_i')}\right)^{it}.
\]
Recall that $S_{i,\ell_i,\delta}$ was chosen so that
\[
\left|\log\frac{P_{i,2}(n_i',m_i')}{P_{i,3}(n_i',m_i')}\right|\le\delta
\]
for $(n_i',m_i')\in S_{i,\ell_i,\delta}$. By $|e^{ix}-1|\le|x|$,
\[
\left|
\left(\frac{P_{i,2}(n_i',m_i')}{P_{i,3}(n_i',m_i')}\right)^{it}-1
\right|^2\le\min\{4,t^2\delta^2\}.
\]
Thus, after taking limits, the inner average in
\eqref{eq:first-shift-spectral-integral} is at most
$\min\{4,t^2\delta^2\}$. By dominated convergence, the integral
tends to zero as $\delta\to0$.

This proves that we may replace
$T_{\frac{P_{i,2}(n_i,m_i)}{P_{i,3}(n_i,m_i)}}F_1$ by $F_1$.

Similarly, we next show that we can replace the second term,
$T_{\frac{P_{i,1}(n_i,m_i)}{\ell_iP_{i,3}(n_i,m_i)}}F_2$, by $F_2$.
Since each factor is one-bounded, the triangle inequality and
Cauchy--Schwarz give
\begin{align*}
&\left|\cesE_{q_{i,1}}\cesE_{q_{i,2}}\cesE_{q_{i,3}}
\cesE_{(n_i',m_i')\in[N]^2\cap S_{i,\ell_i,\delta}}
\int_X H_i F_1\left[T_{\frac{P_{i,1}(n_i,m_i)}{\ell_i P_{i,3}(n_i,m_i)}}F_2-F_2\right]d\mu\right|\\
&\le\cesE_{q_{i,1}}\cesE_{q_{i,2}}\cesE_{q_{i,3}}
\cesE_{(n_i',m_i')\in[N]^2\cap S_{i,\ell_i,\delta}}
\left\|T_{\frac{P_{i,1}(n_i,m_i)}{\ell_i P_{i,3}(n_i,m_i)}}F_2-F_2\right\|_2.
\end{align*}
Since $F_2$ belongs to the compact Archimedean factor, recall that we have
$G_1,\ldots,G_d$ with eigenvalues $g_j(n)=n^{it_j}$ such that
\[
\left\|F_2-\sum_{j=1}^d\lambda_jG_j\right\|_2<\e.
\]
We take the finite sum to be an orthogonal projection of $F_2$, so
\[
\sum_{j=1}^d|\lambda_j|^2\|G_j\|_2^2\le\|F_2\|_2^2\le1.
\]
By the triangle inequality and orthogonality,
\[
\left\|T_{\frac{P_{i,1}(n_i,m_i)}{\ell_i P_{i,3}(n_i,m_i)}}F_2-F_2\right\|_2
\le2\e+\left(\sum_{j=1}^d|\lambda_j|^2\|G_j\|_2^2
\left|\left(\frac{P_{i,1}(n_i,m_i)}{\ell_i P_{i,3}(n_i,m_i)}\right)^{it_j}-1\right|^2\right)^{1/2}.
\]
Recall that $S_{i,\ell_i,\delta}$ was chosen for these frequencies so that
\[
\max_{1\le j\le d}
\left|\left(\frac{P_{i,1}(n_i',m_i')}{\ell_i P_{i,3}(n_i',m_i')}\right)^{it_j}-1\right|
\le\delta.
\]
Since $n_i=Q_i n_i'+v_i$, $m_i=Q_i m_i'+w_i$, and $0\le v_i,w_i<Q_i$, we have
\[
\left|
\log\frac{P_{i,1}(n_i,m_i)}{\ell_i P_{i,3}(n_i,m_i)}
-\log\frac{P_{i,1}(n_i',m_i')}{\ell_i P_{i,3}(n_i',m_i')}
\right|\ll\frac1{n_i'}+\frac1{m_i'}.
\]
Using $|e^{ix}-e^{iy}|\le|x-y|$ and averaging, we obtain
\[
\cesE_{q_{i,1}}\cesE_{q_{i,2}}\cesE_{q_{i,3}}
\cesE_{(n_i',m_i')\in[N]^2\cap S_{i,\ell_i,\delta}}
\left\|T_{\frac{P_{i,1}(n_i,m_i)}{\ell_i P_{i,3}(n_i,m_i)}}F_2-F_2\right\|_2
\le2\e+\delta
+o_{N\to\infty}(1).
\]
Letting $\delta\to0$
leaves an error of at most $2\e$.

We have therefore replaced the shifted integral by
\[
\int_XF(F_1F_2)^r\,d\mu
\]
with an error of $O(\e)$ after averaging and taking the limits.
Chu's inequality \cite[Lemma 1.6]{chu-2-commuting} states that, for
any bounded nonnegative function $H$ on a probability space and
sub-$\sigma$-algebras $\mathcal B_1,\ldots,\mathcal B_k$,
\[
\int_X H\prod_{j=1}^k\E(H\mid\mathcal B_j)\,d\mu
\ge\left(\int_X H\,d\mu\right)^{k+1}.
\]
Applying Chu's inequality with $H=F$, since $F_1$ and $F_2$ are conditional expectations of $F$ and each occurs $r$ times, we obtain
\[
\int_XF(F_1F_2)^r\,d\mu
\geq\left(\int_XF\,d\mu\right)^{2r+1}.
\]

\Needspace{10\baselineskip}
Finally, restoring the $\ell$-average gives
\[
\limsup_{L\to\infty}\cesE_{\ell\in\Psi_L}\ch_B(\ell)
\int_XF(F_1F_2)^r\,d\mu
=\overline d_\Psi(B)\int_XF(F_1F_2)^r\,d\mu.
\]
Including the earlier errors, the original expression is therefore at least
\[
\overline d_\Psi(B)\left(\int_XF\,d\mu\right)^{2r+1}-O_r(\e).
\]
Given $\eta>0$, we first take $\e$ sufficiently small and then
$\delta$ sufficiently small so that the preceding errors total less
than $\eta$. This gives the bound
$\beta\mu(E)^{2r+1}-\eta$ in Theorem~\ref{intro:dynamical}, and since
$\overline d_{\Psi}(B)>0$ and $\int_X F\,d\mu>0$, it also proves the
required positivity.
\end{proof}
\appendix
\numberwithin{equation}{section}
\numberwithin{thm}{section}
\renewcommand{\thecor}{\thethm}
\renewcommand{\thelem}{\thethm}

\section{Other examples}

In this appendix, we give some more examples of equations to which
Theorem~\ref{intro:general} applies.

\subsection{The Equations}

\begin{cor}\label{app:scaled}
Let $a,b\in\N$, and let $A,B\subseteq\N$ have positive
multiplicative upper density. Then there are $x,y,w\in A$ and
$z\in B$ such that
\begin{equation}\label{app:scaled-equation}
a^2y^2+b^2(xz)^2=a^2w^2.
\end{equation}
In particular, this equation is partition regular for every
$a,b\in\N$.
\end{cor}

For $a=b=1$, this is the Pythagorean result stated earlier.

\begin{cor}\label{app:irreducible}
Let $r,s,h,\lambda\in\Z$, where $r\ne0$, $r\ne s$, $\lambda\ne rs$,
and $h^2-4\lambda$ is not a square in $\Z$.
If $A,B\subseteq\N$ have positive multiplicative upper density,
then there are $x,y,w\in A$ and $z\in B$ such that
\begin{equation}\label{app:irreducible-equation}
(w-y-(h+r+s)xz)(\lambda y-rs\,w+(\lambda(r+s)+rsh)xz)
=(\lambda-rs)^2(xz)^2.
\end{equation}
In particular, this equation is partition regular.
\end{cor}

For example, taking $r=1$, $s=-1$, $h=1$, and $\lambda\in\N$
gives the equation
\[
(w-y-xz)(w+\lambda y-xz)=(\lambda+1)^2(xz)^2.
\]
Thus, for any $A,B\subseteq\N$ of positive multiplicative upper
density, there are $x,y,w\in A$ and $z\in B$ satisfying this
equation.

\begin{cor}\label{app:split}
Let $r,s,t,u\in\Z$ satisfy
\[
rt(r-s)(t-u)(r-t)(r-u)(t-s)(tu-rs)\ne0.
\]
If $A,B\subseteq\N$ have positive multiplicative upper density,
then there are $x,y,w\in A$ and $z\in B$ such that
\begin{equation}\label{app:split-equation}
(w-y-(r+s-t-u)xz)(tu\,y-rs\,w+(tu(r+s)-rs(t+u))xz)=(tu-rs)^2(xz)^2.
\end{equation}
In particular, this equation is partition regular.
\end{cor}

For a concrete example, take $r=1$, $s=-1$, $t=b$, and $u=-b-1$,
where $b\geq2$. We obtain $x,y,w\in A$ and $z\in B$ such that
\[
(y-w+xz)(b(b+1)y-w+xz)=(b(b+1)-1)^2(xz)^2.
\]
\subsection{Proofs}

For each equation, we choose a triple of homogeneous polynomials
$(P_1,P_2,P_3)$ to which the general theorem applies. Its conclusion
can be written as
\begin{equation}\label{app:conclusion}
x=\frac{k P_1(n,m)}{z}\in A,\;
y=k P_2(n,m)\in A
\text{ \ \ and \ \ }w=k P_3(n,m)\in A
\end{equation}
for some $k\in\Q_{>0}$, $z\in B$, and $n,m\in\N$.
Indeed, in the notation of Theorem~\ref{intro:general}, we take
$k=a/P_3(n,m)$; the identities below do not require $k$ to be an integer.

The main point is to choose the residues used in the arithmetic progressions
so as to verify condition~\eqref{item:good-congruences} in
Definition~\ref{def:good-multiplicative-recurrence}. The constructions
below follow the same idea as \cite[Lemma~4.2]{pret_pyth}, with
the residues written out for each family. For the modular condition, we must choose a fixed $W$ so that the following holds:
Let $q_1,q_2,q_3$ be admissible pairwise coprime divisors of $Q$
with $W\mid q_1$. When $P_3$ is an irreducible quadratic, we also
have $P_3(1,a)\equiv0\pmod{q_3}$ for some $a\in\Z$.
Then we must show there exist $v,w\in\N$ satisfying the following three conditions:
\begin{equation}\label{app:gcd}
\gcd(P_j(v,w),Q)=q_j\;(j=1,2,3)
\end{equation}
and
\begin{align}
P_1(v,w)&\equiv q_1\pmod{q_1^2}\label{app:congruences}\\
\intertext{and}
P_2(v,w)&\equiv P_3(v,w)\equiv1\pmod{q_1}.\label{app:congruences-second}
\end{align}

In fact, in all three examples, we take $v=1$ and find $w$
satisfying \eqref{app:gcd}, \eqref{app:congruences}, and
\eqref{app:congruences-second}. We will choose $w$ separately at each prime dividing $Q$.
Since $q_1,q_2,q_3$ are pairwise coprime, the conditions at each
prime depend on which $q_i$ it divides, if any. After finding a
suitable congruence class at each prime power, we use the Chinese
remainder theorem to obtain a positive integer $w$ satisfying
all the conditions.

We first record an explicit formula for the square roots that
occur in two of the constructions. If $u\equiv1\pmod p$, then
\begin{equation}\label{app:power}
u^{p^e}\equiv1\pmod{p^{e+1}}\;(e\geq1),
\end{equation}
by the binomial expansion.
Consequently, if $p$ is odd, $p\nmid D$, and
$\sigma^2\equiv D\pmod p$, then
\begin{equation}\label{app:root}
R=\sigma^{p^e}D^{-(p^e-1)/2}
\text{ \ \ satisfies \ \ }
R^2\equiv D\pmod{p^{e+1}}
\text{ \ \ and \ \ }R\equiv\sigma\pmod p.
\end{equation}
After all, $R^2=D(\sigma^2D^{-1})^{p^e}$, so the first
congruence follows directly from \eqref{app:power}.

The polynomials we will choose have the form given below. We therefore
need the following lemma.

\begin{lem}[The sector condition]\label{app:sector-lemma}
Let $c>0$ and $\beta,\gamma,h,\lambda\in\R$, and let
\[
P_1(n,m)=cnm,\; P_2(n,m)=n^2+\beta nm+\gamma m^2
\text{ \ \ and \ \ }P_3(n,m)=n^2+hnm+\lambda m^2.
\]
There is a $\kappa>0$ such that $P_j(x,1)>0$ for $x>\kappa$
and $j=1,2,3$. For every finite nonempty set $\Theta\subseteq\R$
and every $\delta>0$, there are constants $\kappa<a<b<\infty$
and $\eta_0>0$ such that, for each $\ell\in\N$, there is an open
interval $J_{\ell,\delta}\subseteq[a,b]$ of length at least $\eta_0$
on which
\[
\left|\log\frac{P_2(x,1)}{P_3(x,1)}\right|\leq\delta
\text{ \ \ and \ \ }
\max_{t\in\Theta}\left|\left(\frac{P_1(x,1)}{\ell P_3(x,1)}\right)^{it}-1\right|\leq\delta.
\]
The constants $a,b,\eta_0$ are independent of $\ell$.
\end{lem}

\begin{proof}
Choose $\kappa$ so that all three polynomials are
positive for $n>\kappa m$. As in the Pythagorean case, the two
ratios we need to control are
\[
\frac{P_1(x,1)}{\ell P_3(x,1)}
\text{ \ \ and \ \ }\frac{P_2(x,1)}{P_3(x,1)}.
\]
The second tends to $1$ as $x\to\infty$. For the first, writing
$x=e^v$ gives
\begin{equation}\label{app:log-ratio}
\log\frac{P_1(e^v,1)}{\ell P_3(e^v,1)}
=\log c-\log\ell-v-\log(1+he^{-v}+\lambda e^{-2v}).
\end{equation}

It suffices to consider $0<\delta\leq1$. Fix a finite nonempty
set of frequencies $\Theta\subset\R$ and such a $\delta$, and put
$T=\max(\{1\}\cup\{|t|:t\in\Theta\})$.
Choose $v_0$ large enough that $e^{v_0}>\kappa$ and, for every
$v\geq v_0$,
\[
\left|\log\frac{P_2(e^v,1)}{P_3(e^v,1)}\right|\leq\delta
\text{ \ \ and \ \ }
T\left|\log(1+he^{-v}+\lambda e^{-2v})\right|\leq\delta/4.
\]
It remains to arrange the finitely many phases coming from
$\log c-\log\ell-v$, with $v$ in a bounded interval independent
of $\ell$.

For this, put $\phi(v)=(e^{itv})_{t\in\Theta}$ and let $G$ be
the closure of $\phi(\R)$ in the torus. By compactness, there
is $S>0$ such that $\phi([-S,S])$ is $\delta/4$-dense in $G$,
using the maximum of the coordinate distances. Every translate
of this interval has the same property, since translation of
$v$ multiplies $\phi(v)$ by an element of $G$. In particular,
for each $\ell$ we can choose
$v_\ell\in[v_0+1,v_0+1+2S]$ such that
\[
\max_{t\in\Theta}
\left|e^{it(\log c-\log\ell-v_\ell)}-1\right|<\delta/4.
\]

Put
\[
\eta=\min\{1/2,\delta/(4T)\}
\text{ \ \ and \ \ }J_\ell=(e^{v_\ell-\eta},e^{v_\ell+\eta}).
\]
For $x=e^v\in J_\ell$, the triangle inequality and
\eqref{app:log-ratio} give
\[
\left|\left(\frac{P_1(x,1)}{\ell P_3(x,1)}\right)^{it}-1\right|
\leq\delta/4+T|v-v_\ell|+\delta/4\leq3\delta/4
\;(t\in\Theta).
\]
The other logarithmic ratio is at most $\delta$ by the choice of
$v_0$. Moreover, all these intervals lie in
$[e^{v_0},e^{v_0+2S+2}]$, and each has length at least
$2\eta e^{v_0}$. Thus both the containing interval and the lower
bound on the lengths are independent of $\ell$, as required.
\end{proof}

\begin{proof}[Proof of Corollary~\ref{app:scaled}]
We use the polynomials
\[
P_1(n,m)=2anm,\; P_2(n,m)=n^2-b^2m^2
\text{ \ \ and \ \ }P_3(n,m)=n^2+b^2m^2.
\]
Take $W=120a^2b$. The first two polynomials split into the pairwise nonproportional
linear factors $n,m,n-bm,n+bm$, and the third is an irreducible
quadratic. Thus the factorization conditions hold.

Fix admissible $q_1,q_2,q_3$ and $Q$. If $p^e\Vert q_3$ (that is, $\nu_p(q_3)=e\geq1$), then
$p\equiv1\pmod4$, because \eqref{eq:good-root-condition} gives
a solution to $1+b^2x^2\equiv0\pmod p$. Put
\[
\iota_p=\left(\frac{p-1}{2}\right)!
\text{ \ \ and \ \ }I_{p,e}=\iota_p^{p^e}.
\]
We claim that $(p-1)!\equiv-1\pmod p$. After all,
$(p-1)!$ contains the product of every number and its multiplicative inverse,
except for $1$ and $p-1\equiv-1\pmod p$, which are their own
inverses and have product $-1\pmod p$. We also claim that $\iota_p^2\equiv-1\pmod p$. After all, since
$(p-1)/2$ is even, since $p\equiv1\pmod4$, we have
\[
\iota_p^2=\left(\frac{p-1}{2}\right)!\prod_{j=1}^{(p-1)/2}j
\equiv\left(\frac{p-1}{2}\right)!\prod_{j=1}^{(p-1)/2}(p-j)
=(p-1)!\equiv-1\pmod p.
\] By \eqref{app:power},
$I_{p,e}^2\equiv-1\pmod{p^{e+1}}$.

Choose $w$ to satisfy
\begin{equation}\label{app:scaled-residues}
\begin{aligned}
w&\equiv q_1/(2a) &&\pmod{q_1^2/(2a)},\\
w&\equiv b^{-1}(1-p^e) &&\pmod{p^{e+1}} &&\text{ \ \ for \ \ }(p^e\Vert q_2),\\
w&\equiv b^{-1}(I_{p,e}+p^e) &&\pmod{p^{e+1}} &&\text{ \ \ for \ \ }(p^e\Vert q_3),\\
w&\equiv 2b^{-1} &&\pmod p &&\text{ \ \ for \ \ }(p\mid Q,\ p\nmid q_1q_2q_3).
\end{aligned}
\end{equation}
The quotient in the first row is an ordinary integer quotient:
$2a\mid q_1$. Every prime in the other rows is coprime to
$60ab$, so all the indicated inverses exist.

The first row gives $2aw\equiv q_1\pmod{q_1^2}$. If
$p^e\Vert q_1$ and $c=\nu_p(2a)$, then $e\geq2c$ by definition of $W$, and
$\nu_p(w)=e-c$. In particular, $q_1\mid b^2w^2$, and hence
\[
P_1(1,w)\equiv q_1\pmod{q_1^2}
\text{ \ \ and \ \ }P_2(1,w)\equiv P_3(1,w)\equiv1\pmod{q_1}.
\]
This gives the required valuations at the primes dividing
$q_1$.

If $p^e\Vert q_2$, then the second row gives
\[
1-b^2w^2\equiv2p^e\pmod{p^{e+1}}
\text{ \ \ and \ \ }1+b^2w^2\equiv2\pmod p.
\]
Thus $\nu_p(P_2(1,w))=e$, while $P_1(1,w)$ and $P_3(1,w)$
are nonzero modulo $p$. Likewise, if $p^e\Vert q_3$, then
\[
1+b^2w^2\equiv2I_{p,e}p^e\pmod{p^{e+1}}
\text{ \ \ and \ \ }1-b^2w^2\equiv2\pmod p.
\]
This gives $\nu_p(P_3(1,w))=e$, with the other two values
nonzero modulo $p$. Finally, at a prime in the last row of
\eqref{app:scaled-residues}, the three values are
\[
P_1(1,w)\equiv4ab^{-1},\; P_2(1,w)\equiv-3
\text{ \ \ and \ \ }P_3(1,w)\equiv5\pmod p.
\]
They are all nonzero. This proves \eqref{app:gcd} and
\eqref{app:congruences} and \eqref{app:congruences-second}.

The polynomials are positive for $n>bm$, and the sector condition
follows from Lemma~\ref{app:sector-lemma}.
Finally,
\[
a^2P_2(n,m)^2+b^2P_1(n,m)^2=a^2P_3(n,m)^2.
\]
Substituting \eqref{app:conclusion} gives
\eqref{app:scaled-equation}.
\end{proof}

\begin{proof}[Proof of Corollary~\ref{app:irreducible}]
This time we take
\[
P_1(n,m)=nm,\; P_2(n,m)=(n-rm)(n-sm)
\text{ \ \ and \ \ }P_3(n,m)=n^2+hnm+\lambda m^2,
\]
and put $D=h^2-4\lambda$.
The discriminant condition makes $P_3$ irreducible. The first
two polynomials split into the linear factors $n$, $m$, $n-rm$, and $n-sm$.
This gives the factorization conditions.

Let
\[
c=2+|r|+|s|+|h|+|\lambda|,
\]
and choose
\[
W=\left|30r(r-s)\lambda D P_3(r,1)P_3(s,1)cP_2(c,1)P_3(c,1)\right|.
\]
Every factor here is nonzero because $P_3$ has no rational
root and our choice of $c$ makes both $P_2(c,1)$ and $P_3(c,1)$
positive.

Fix admissible $q_1,q_2,q_3$ and $Q$. If $p^e\Vert q_3$, then
\eqref{eq:good-root-condition} gives a solution to $P_3(1,a)\equiv0\pmod p$. Since $p\nmid2\lambda D$,
this gives a nonzero square root of $D$ modulo $p$; indeed,
\[
(2\lambda a+h)^2-D=4\lambda P_3(1,a).
\]
Let $\sigma_p$ be the least positive square root of $D$ modulo
$p$, and put
\[
R_{p,e}=\sigma_p^{p^e}D^{-(p^e-1)/2}\pmod{p^{e+1}}.
\]
By \eqref{app:root}, this is a square root of $D$ modulo
$p^{e+1}$. Choose $w$ to satisfy
\[
\begin{aligned}
w&\equiv q_1 &&\pmod{q_1^2},\\
w&\equiv r^{-1}(1-p^e) &&\pmod{p^{e+1}} &&\text{ \ \ for \ \ }(p^e\Vert q_2),\\
w&\equiv (R_{p,e}-h)(2\lambda)^{-1}+p^e
&&\pmod{p^{e+1}} &&\text{ \ \ for \ \ }(p^e\Vert q_3),\\
w&\equiv c^{-1} &&\pmod p &&\text{ \ \ for \ \ }(p\mid Q,\ p\nmid q_1q_2q_3).
\end{aligned}
\]

The first row gives \eqref{app:congruences} and \eqref{app:congruences-second} immediately. At a
prime with $p^e\Vert q_2$, the second row gives
\[
1-rw\equiv p^e\pmod{p^{e+1}}
\text{ \ \ and \ \ }1-sw\equiv1-sr^{-1}=(r-s)r^{-1}\pmod p.
\]
The second factor is nonzero $\pmod p$ by choice of $W$, since
$p\nmid r(r-s)$, so $\nu_p(P_2(1,w))=e$. Also, $w$ is nonzero $\pmod p$, and
\[
P_3(1,w)\equiv r^{-2}P_3(r,1)\not\equiv0\pmod p.
\]

For a prime with $p^e\Vert q_3$, by completing the square, we have
\[
4\lambda P_3(1,w-p^e)
\equiv R_{p,e}^2-D\equiv0\pmod{p^{e+1}},
\]
and adding $p^e$ to this root gives
\[
P_3(1,w)\equiv R_{p,e}p^e\pmod{p^{e+1}}.
\]
Since $R_{p,e}$ is nonzero $\pmod p$, this gives
$\nu_p(P_3(1,w))=e$. Also, $w$ is nonzero $\pmod p$, since
$P_3(1,0)=1$. Neither factor of $P_2(1,w)$ can vanish $\pmod p$: a root $w=r^{-1}$ would give $P_3(r,1)\equiv0\pmod p$,
and, if $p\nmid s$, a root $w=s^{-1}$ would give
$P_3(s,1)\equiv0\pmod p$. If $p\mid s$, the latter factor is
simply $1$ $\pmod p$.

Finally, for a prime in the last row, homogeneity gives
\[
P_j(1,w)\equiv c^{-2}P_j(c,1)\not\equiv0\pmod p
\;(j=1,2,3).
\]
We know $c$ is invertible by the definition of $W$. Thus,
as before, the congruences give \eqref{app:gcd} as well as
\eqref{app:congruences} and \eqref{app:congruences-second}. Positivity holds when $n/m$ is
sufficiently large, and the sector condition follows from
Lemma~\ref{app:sector-lemma}.

The identities
\[
P_3-P_2-(h+r+s)P_1=(\lambda-rs)m^2
\]
and
\[
\lambda P_2-rsP_3+(\lambda(r+s)+rsh)P_1=(\lambda-rs)n^2
\]
give \eqref{app:irreducible-equation} after multiplying and
substituting \eqref{app:conclusion}.
\end{proof}

\begin{proof}[Proof of Corollary~\ref{app:split}]
We use
\[
P_1(n,m)=nm,\; P_2(n,m)=(n-rm)(n-sm)
\text{ \ \ and \ \ }P_3(n,m)=(n-tm)(n-um).
\]
The polynomials split into the linear factors $n$, $m$, $n-rm$,
$n-sm$, $n-tm$, and $n-um$. This gives the factorization conditions.

Similarly to the proof of Corollary~\ref{app:irreducible}, let
\[
c=1+|r|+|s|+|t|+|u|,
\]
and choose
\[
W=\left|30rt(r-s)(t-u)(r-t)(r-u)(t-s)cP_2(c,1)P_3(c,1)\right|.
\]
For admissible $q_1,q_2,q_3$ and $Q$, choose $w$ to satisfy
\[
\begin{aligned}
w&\equiv q_1 &&\pmod{q_1^2},\\
w&\equiv r^{-1}(1-p^e) &&\pmod{p^{e+1}} &&\text{ \ \ for \ \ }(p^e\Vert q_2),\\
w&\equiv t^{-1}(1-p^e) &&\pmod{p^{e+1}} &&\text{ \ \ for \ \ }(p^e\Vert q_3),\\
w&\equiv c^{-1} &&\pmod p &&\text{ \ \ for \ \ }(p\mid Q,\ p\nmid q_1q_2q_3).
\end{aligned}
\]
The first and last rows work exactly as in the previous example.
If $p^e\Vert q_2$, then
\[
1-rw\equiv p^e\pmod{p^{e+1}}
\text{ \ \ and \ \ }1-sw\equiv1-sr^{-1}=(r-s)r^{-1}\pmod p,
\]
whereas
\[
P_3(1,w)\equiv(r-t)(r-u)r^{-2}\not\equiv0\pmod p.
\]
By the definition of $W$, $p\nmid r(r-s)(r-t)(r-u)$. Thus
$\nu_p(P_2(1,w))=e$, while $P_1(1,w)$ and $P_3(1,w)$ are
nonzero $\pmod p$. Similarly, if $p^e\Vert q_3$, then
\[
1-tw\equiv p^e\pmod{p^{e+1}}
\text{ \ \ and \ \ }1-uw\equiv1-ut^{-1}=(t-u)t^{-1}\pmod p,
\]
whereas
\[
P_2(1,w)\equiv(t-r)(t-s)t^{-2}\not\equiv0\pmod p.
\]
By the definition of $W$, $p\nmid t(t-u)(t-r)(t-s)$. Thus
$\nu_p(P_3(1,w))=e$, while $P_1(1,w)$ and $P_2(1,w)$ are
nonzero $\pmod p$. This proves the modular condition. The remaining positivity
and sector conditions follow from Lemma~\ref{app:sector-lemma}.

As in the previous case,
\[
P_3-P_2-(r+s-t-u)P_1=(tu-rs)m^2
\]
and
\[
tuP_2-rsP_3+(tu(r+s)-rs(t+u))P_1=(tu-rs)n^2.
\]
Multiplying and substituting \eqref{app:conclusion} gives
\eqref{app:split-equation}.
\end{proof}

The same verification applies to any finite collection of these
examples. For example, we have the following corollary.

\begin{cor}\label{app:simultaneous-examples}
Let $a,b,\lambda\in\N$ and let $k\geq2$ be an integer. If
$A\subseteq\N$ and $B\subseteq\N^3$ have positive multiplicative
upper density, then there are $w,x_1,y_1,x_2,y_2,x_3,y_3\in A$
and $(z_1,z_2,z_3)\in B$ such that
\[
a^2y_1^2+b^2(x_1z_1)^2=a^2w^2,
\]
\[
(w-y_2-x_2z_2)(w+\lambda y_2-x_2z_2)
=(\lambda+1)^2(x_2z_2)^2,
\]
and
\[
(y_3-w+x_3z_3)(k(k+1)y_3-w+x_3z_3)
=(k(k+1)-1)^2(x_3z_3)^2.
\]
In particular, every finite coloring of $\N$ admits a monochromatic
solution to these three equations with the same value of $w$.
\end{cor}

\begin{proof}
Apply Theorem~\ref{intro:general} to the three polynomial triples
used for these examples in Corollaries~\ref{app:scaled},
\ref{app:irreducible}, and~\ref{app:split}.
\end{proof}

We can also iterate the Pythagorean result, using each hypotenuse
as a leg in the next triple.

\begin{cor}\label{app:iterated-pythagorean}
Let $A,B\subseteq\N$ have positive multiplicative upper density and
let $r\in\N$. Then there are
$y_0,\ldots,y_r,z_1,\ldots,z_r\in A$ and $x_1,\ldots,x_r\in B$ such that
\[
y_{i-1}^2+(x_i z_i)^2=y_i^2\;(i=1,\ldots,r).
\]
\end{cor}

\begin{proof}
We prove by induction that the set of possible final values $y_r$
has positive multiplicative upper density. For $r=0$, this set is
$A$. Suppose the claim holds for $r-1$, and let $C\subseteq A$ be
the set of possible final values of chains of length $r-1$.
Let $D\subseteq C$ consist of the values that cannot be the final
value of a chain of length $r$.

If $D$ had positive multiplicative upper density, then
Theorem~\ref{intro:density}, applied to $D$ and $B$, would give
$u,z,y\in D$ and $x\in B$ such that
\[
u^2+(xz)^2=y^2.
\]
Since $u\in C$, we can append this equation to a chain of length
$r-1$ ending at $u$, giving a chain of length $r$ ending at $y$.
This contradicts $y\in D$. Thus $D$ has zero multiplicative upper
density, so $C\setminus D$ has positive multiplicative upper
density. This proves the induction step.
\end{proof}

\providecommand{\bysame}{\leavevmode\hbox to3em{\hrulefill}\thinspace}
\providecommand{\MR}{\relax\ifhmode\unskip\space\fi MR }
\providecommand{\MRhref}[2]{%
  \href{http://www.ams.org/mathscinet-getitem?mr=#1}{#2}
}
\providecommand{\href}[2]{#2}

\end{document}